\documentclass[10pt]{amsart} 

\usepackage{amsmath, amsthm, amsfonts, amssymb}

\usepackage{amsrefs}
\usepackage{hyperref}
\usepackage{cleveref}
\usepackage[all]{xy}
\usepackage{mathrsfs}
\usepackage{enumerate}

\theoremstyle{plain}
\newtheorem{teo}{Theorem}[section]
\newtheorem{prop}[teo]{Proposition}
\newtheorem{cor}[teo]{Corollary}
\newtheorem{lema}[teo]{Lemma}
\theoremstyle{definition}
\newtheorem{defs}[teo]{Definition}

\newtheorem{rmk}[teo]{Remark}

\def\Z{\mathbb{Z}}
\def\C{\mathbb{C}}
\def\N{\mathbb{N}}
\def\Q{\mathbb{Q}}
\def\PP{\mathbb{P}}

\def\F{\mathcal F}

\def\O{\mathcal O}
\def\spec{\mathrm{Spec}}
\def\epi{\twoheadrightarrow}
\def\-#1{\overline{#1}}
\def\^#1{\widehat{#1}}
\def\sing{\mathrm{sing}}

\date{} 

\begin{document}

\title{Families of distributions and Pfaff systems under duality}
\author{Federico Quallbrunn}
\email{fquallb@dm.uba.ar}
\address{Departamento de Matem\'atica, FCEyN, Universidad de Buenos Aires, Ciudad Universitaria, Pabell\'on 1, Buenos Aires (Argentina)}

\begin{abstract}
A singular distribution on a non-singular variety $X$ can be defined either by a subsheaf $D\subseteq TX$ of the tangent sheaf, or by the zeros of a subsheaf $D^0\subseteq \Omega^1_X $ of $1$-forms, that is a Pfaff system.
Although both definitions are equivalent under mild conditions on $D$, they give rise, in general, to non-equivalent notions of flat families of distributions.
In this work we investigate conditions under which both notions of flat families 
are equivalent.
In the last sections we focus on the case where the distribution is \emph{integrable}, and we use our results to generalize a theorem of Cukierman and Pereira.
\end{abstract}

\maketitle

\section{Introduction}

Among the several motivations for studying moduli problems in geometry there is the study of the topological and geometrical properties 
of differential equations on manifolds.
Such problems are studied in the works of R.Thom,  V.Arnold and Kodaira-Spencer in the 1950's and were an influence in the work of Zariski about deformation of singularities, we refer to the introduction and section 3.10 of \cite{Dema} for a historical account on the subject. 
Algebraic geometric techniques where used in the study of (integrable) differential equations on non-singular algebraic varieties at least since Jouanolou \cite{Jou}, moduli and deformation problems of integrable differential equations where studied by Gomez-Mont in \cite{GM}.
Since then much have been done in the study of the geometry of integrable $1$-forms on algebraic varieties, specially describing singularities of the differential equations they define (see, for instance \cite{Med}). 
Also, much has been done in the determination of irreducible components of the moduli space of integrable $1$-forms (e.g.: \cite{split, rational, LN} and references therein).
\par
Many of these works where carried in a `na\" ive' way, without concerning for representability of functors or formalization of deformation problems. 
When trying to formalize the moduli problem for integrable differential equations in non-singular algebraic varieties, an issue appears at the very beginning: although its equivalent to describe a distribution by its tangent sheaf or by the sheaf of $1$-forms that vanishes on it, both descriptions give rise to inequivalent notions of a flat family of distributions.
Thus we are led to two different moduli problems, one for involutive) sub-sheaves of the tangent sheaf, and the other for (integrable) sub-sheaves of the sheaf of $1$-forms.
In the present work we prove in \cref{proprep1} and \cref{proprep2} that, under suitable assumptions, these two moduli problems are representable  and, in \cref{corobir}, that their irreducible components are moreover birationally equivalent. 
Using this we can generalize previous theorems in the literature as Theorem 1 of \cite{split} and the theorems of \cite{P, P2}.

\medskip
In order to better explain our results we now introduce some notation and definitions.
\newline Let $X$ be a non-singular projective algebraic variety.
A non-singular $k$-dimensional distribution $D$ on a non-singular variety $X$ consist on a $k$-dimensional subspace $D_x\subseteq T_xX$ of the tangent space of $X$ at $x$ varying continuously with $x$. 
This notion can be formalized by saying that a $k$-dimensional distribution is a rank $k$ sub-bundle of the tangent bundle $D\hookrightarrow TX$, and taking $D_x\subseteq T_xX$ to be $D\otimes k(x)$. 
Equivalently, we can say that a $k$-dimensional distribution is determined by a sub-bundle $I_D\hookrightarrow \Omega^1_X$, and take $D_x$ to be $\{ v\in T_xX\ |\ \omega(v)=0, \forall \omega\in I_D\otimes k(x)\}$.

\par When $X$ is an algebraic variety, it is often the case that there are no algebraic sub-bundles of $TX$ or $\Omega^1_X$ of a given rank i.e.: there are no \emph{non-singular} distributions.
Nevertheless, the definition can be readily generalized to allow $D\subseteq TX$ and $I_D\subseteq \Omega^1_X$ to be subsheaves. 
In this way, we describe distributions $D$ which are $k$-dimensional on a dense open subset, but may present singularities along proper subvarieties.
Again, a singular distribution is equivalently defined either by a subsheaf $D\subseteq TX$ or by its annihilator  $I_D\subseteq \Omega^1_X$.
However, as was already observed by Pourcin in \cite{P}, when studying \emph{families} of distributions parametrized by a base scheme $S$ it may happen that while a family of distributions $D\subseteq T_S(X\times S)$ is flat (in the sense that the quotient  $T_S(X\times S)/D$ is a flat sheaf over $S$), its annihilator $I_D\subseteq \Omega^1_{X|S}$ may not be flat.
This gives us two different notions of flat family of (singular) distributions, and therefore two different moduli problems for them.

\par In this work we prove that the above two notions of flatness coincide as long as the singular set of the distribution (endowed with a convenient scheme structure) is also flat over $S$ (\cref{propdual1} and \cref{propdual2}).
We focus on the case of \emph{integrable} distributions, in \cref{proprep1} we give constructions for the moduli space $\mathrm{Inv}^X$ of involutive (in the sense of the Frobenius theorem) subsheaves of $TX$ and in \cref{proprep2} for the moduli space $\mathrm{iPf}^X$ of integrable subsheaf of $\Omega^1_X$ and conclude in \cref{corobir} that taking annihilators defines  a rational map between $\mathrm{Inv}^X$  and $\mathrm{iPf}^X$ that is a birational equivalence in each irreducible component of  $\mathrm{Inv}^X$.
Moreover we give in \cref{papa} a sufficient criterion in terms of the singularity of the foliation to know when an involutive subsheaf $T\F\subseteq TX$ represents a point in the dense open set $U\subseteq \mathrm{Inv}^X$ where taking duals gives an isomorphism with an open set $V\subseteq \mathrm{iPf}^X$. Using this criterion we can generalize the main theorem of \cite{split} by specializing our results to the case where $X=\PP^n$ and $T\F\cong \bigoplus_i \O_{\PP^n}(d_i)$.
\newline
\par In \Cref{example} we show a particular example of a flat family of Pfaff systems being dual to a non-flat family of distributions.
\par In \Cref{prel} we treat some preliminary general notions on sheaves and criterion for flatness that will be useful later. 
\par In \Cref{duales} we study the effect of applying the functor $\mathcal{H}om(-, \O_X)$ to a short exact sequence of sheaves. We also include in this section some observations on exterior powers, that are relevant for the study of distributions of codimension higher than $1$.
 \par \Cref{distypfaff} consists mainly of definitions of families of distributions and Pfaff systems, and related notions.
\par In \Cref{univ} the construction of the moduli spaces of involutive distributions and integrable Pfaff systems is given as subschemes of certain $\mathrm{Quot}$ schemes.
\par In \Cref{duality} the main results of the paper are proven. 
First the singular scheme of a family of distributions is defined, as well as the analogous notion for family of Pfaff systems and it is proven that the singular scheme of a family of distributions is the same as the singular scheme of its dual family (which is a family of Pfaff systems). 
In \Cref{codim1case} the codimension $1$ version of the main result is proven: if the singular scheme of a flat family of codimension-$1$ Pfaff systems is itself flat then the dual family is flat as well.
In \Cref{codimk} an analogous statement is proven for arbitrary codimension. 
In this case, however, flatness of the singular scheme is not enough to assure flatness of the dual family. 
To obtain a valid criterion we define a stratification of the singular scheme, if each stratum is flat over the base we can assure the dual family will be flat as well.
\par In \Cref{singi} we give a sufficient condition to know when the singular scheme of a family of codimension $1$  foliations is flat over the base. 
This condition is related with two of the better-studied types of singularities of foliations, the Kupka singularities and the Reeb singularities.
We prove that if the singularities of a foliation given by a distribution $T\F\hookrightarrow TX$ are only of this two types, then every flat family 
\[ 0\to T\F_S \to T_S(X\times S) \to N_\F\to 0\]
such that there is an $s\in S$ with $T\F_s=T\F$, is such that $\mathrm{sing}(\F)$ is flat in a neighborhood of $s$.
\par In \Cref{apps} we apply the theorem of \Cref{singi} to recover the main result of \cite{split} as a special case of \cref{teo}, where $X=\PP^n$ and $T\F$ splits as direct sum of line bundles.
\newline
\par The content of this work is part of the author's doctoral thesis, for the degree of Doctor de la Universidad de Buenos Aires, under the advice of Fernando Cukierman. 
The author was supported by a doctoral grant of CONICET.
The author is grateful to Fernando Cukierman and Fernando Sancho de Salas for useful ideas and to Aroldo Kaplan, Alicia Dickenstein and the anonymous referee for a thorough reading of this paper and useful suggestions.
The author is also grateful to Universidad de Buenos Aires, where the work was made. 

\section{An example}\label{example}

Set $X=\C^4$, $S=\C^2$ and consider the family of $1$-forms on $X$

\[ 
\omega_{(t,s)}=xdx+ydy+t\cdot zdz+ s\cdot wdw,
\]
parametrized by $(t,s)\in S=\C^2$, and the related family:
\[I=I_{(t,s)}=\O_{X\times S}\cdot(\omega_{(t,s)})\subseteq \Omega^1_{X\times S|S},\]
of Pfaff systems on $X=\C^4$.

It is easy to see that the sheaf $I$ is flat over $S=\C^2$.
On the other hand if we look at the family of distributions this Pfaff systems define we will find that it is not flat.
Indeed, let $D$ denote the annihilator of $I$.
\[D=\mathrm{Ann}(I)\subseteq T(X\times S).\]
To see that $D$ is \emph{not} flat over $S$ we can use Artin's criterion for flatness \cite{artin}*{Corollary to Proposition 3.9}.
\begin{prop}[Artin's criterion for flatness.]
Let $(A,\mathscr{M})$ and $(B,\mathscr{N})$ be local rings, $B\to A$ a flat morphism of local rings, and $M$ a finitely generated $A$-module. 
Suppose we have generators of $M$, $M=(f_1,\dots, f_r)$. 
Let $\-{f_i}$ be the class of $f_i$ on $M\otimes_B B/\mathscr{N}$.
Then $M$ is flat over $B$ if and only if every relation among $(\-{f_1},\dots, \-{f_r})$ lifts to a relation among $(f_1,\dots,f_r)$.
\end{prop}
Note that, while this is not exactly the same context in which Artin states the proposition, his proof extends mutatis mutandi to this setting.
\smallskip
\par Now observe that $D=\mathrm{Ann}(I)$ has $6$ generators:
\begin{eqnarray*}
f_1=y\frac{\partial}{\partial x} - x\frac{\partial}{\partial y}, & f_2=tz\frac{\partial}{\partial x} - x\frac{\partial}{\partial z}, &f_3=sw\frac{\partial}{\partial x}-x\frac{\partial}{\partial w},\\
 f_4=tz\frac{\partial}{\partial y}-y\frac{\partial}{\partial z}, &f_5=sw\frac{\partial}{\partial y}-y\frac{\partial}{\partial w}, &f_6=sw\frac{\partial}{\partial z}-tz\frac{\partial}{\partial w}.
\end{eqnarray*}
These generators have, in turn, numerous relations between them, too many to list completely here. 
We note, however, that all relations are generated by relations of the shape $af_i+bf_j+cf_k=0$, where $a$, $b$ and $c$ are in the ideal $(x,y,tz,sw)\subseteq \O_{X\times S}$.
We apply Artin's criterion with $A=\O_{X\times S, (0,0)}$, $B=\O_{S, 0}$ and $M=D_{(0,0)}$.
In this case the $\-{f_i}$'s consist of evaluating the $f_i$ on $(x,y,z,w, 0,0)\in \C^4\times (0,0)$.  
We get in this way a particular relation on $M\otimes_B B/\mathscr{N}$, namely $\-{f_6}=0$.
It is then easy to see that this relation does not lift to a relation on $(f_1,\dots,f_6)$.
\smallskip
\par This example shows that there is, in general, no morphism between the moduli spaces $\mathrm{Quot}(X,TX)$ and $\mathrm{Quot}(X,\Omega^1_X)$ induced by taking annihilators. 
We will show ahead that taking annihilators does define a \emph{birational} equivalence between irreducible components of $\mathrm{Quot}_X(TX)$ and those of $\mathrm{Quot}_X(\Omega^1_X)$.
\par The example also hints that a key aspect to understand whether flatness is preserved or not under duality is to look at how the singular set varies. 
Indeed, we can see that  $D$ is flat over $S\setminus \{(0,0)\}$. 
For each point $(t,s)\in S\setminus \{(0,0)\}$, the singular set of $I|_{(t,s)}$ (i.e.: the set of points in $p\in X$ such that $\omega_p=0$) is just the origin of $\C^4=X$.
On the other hand, the singular set of $I|_{(0,0)}$ is the plane $\{(0,0,z,w): z,w \in \C\}\subset \C^4$.
So $D$ ceases to be flat exactly when the singular set of $I$ cease to be flat as well.
This illustrates the typical behavior of singular codimension $1$ Pfaff systems and the distributions they determine.
In the arbitrary codimension case the situation is a bit more subtle, but the singular set still plays a decisive role.
Of course, to make sense out of this we have to define a scheme structure on the singular set of a Pfaff system.
This will be done in the course of the present work.

\section{Preliminaries}\label{prel}

Here we gather known facts of algebraic geometry that will be used later.
We include proofs of some of these facts for lack (to the author best knowledge) of a better reference. 

\subsection{Reflexive sheaves and Serre's property $S_2$}

Property $S_2$ can be viewed as an algebraic analog of Hartog's theorem on complex holomorphic functions. 
For this reason it will be extremely useful to us, for it will allow us to conclude global statements on sheaves that holds, a priori,  for the restriction of this sheaves to (suitably large) open sets. Here we remind some known facts about sheaves with the $S_2$ property, and sheaves with the \emph{relative} $S_2$ property as defined in \cite{Bar}. 

\begin{defs}
Let $p: X\to T$ be a morphism of schemes and for any point $x\in X$ set $d_T(x)$ the codimension  of $x$ in its fiber over $T$.
We say that a sheaf $\mathscr{F}$ on $X$ satisfies the \emph{relative Serre's condition} $S_k$ with respect to $p$ if and only if
\[ \mathrm{depth} \mathscr{F}_{x} \geq \min(k, d_T(x)). \]
for all $x\in X$.
\end{defs}

The proof of the next proposition works exactly as in the non-relative version.

\begin{prop}
Let $p:X\to T$ be a morphism of noetherian scheme and $\mathscr{F}$ a torsion-free coherent sheaf with relative property $S_2$ with respect to $p$.
Let $Y\subset X$ be a closed subset such that $d_T(Y)\geq 2$.
Then the restriction map $\rho:\Gamma(X,\mathscr{F})\to \Gamma(X\setminus Y, \mathscr{F})$ is an isomorphism. 
\end{prop}

\begin{cor}
Let $p:X\to T$ be a morphism of  noetherian schemes and $\mathscr{F}$ a torsion-free coherent sheaf with property $S_2$ with respect to $p$.
Let $Y\subset X$ be a closed subset such that $d_T(Y)\geq 2$.
Denote $U=X\setminus Y$ and $j:U\to X$ the inclusion.
Then $\mathscr{F}= j_*(\mathscr{F}|_U)$.
\end{cor}

This corollary motivates the following definition, which is the relative analogue of a notion due to Grothendieck \cite{EGA4}*{5.10}.

\begin{defs}
Let $p: X\to T$ be a morphism of noetherian scheme and $\mathscr{F}$ a coherent sheaf.
If for each closed subset $Y\subset X$ such that $d_T(Y)\geq 2$, with  $U=X\setminus Y$ and $j:U\to X$ the inclusion, the natural map
\[\rho_U: \mathscr{F}\to j_*(\mathscr{F}|_U)\]
is an epimorphism we say  $\mathscr{F}$ is \emph{$Z^{(2)}$-closed relative to $p$}, if it is an isomorphism we say it is \emph{$Z^{(2)}$-pure relative to $p$}.
\end{defs}

%

%

\begin{prop}[\cite{Hart2}*{Proposition 1.7}]
Let X be a quasi-projective integral scheme. 
A coherent sheaf $\mathscr{F}$ is reflexive if and only if it can be included in an exact sequence
\[ 0\to \mathscr{F}\to \mathscr{E}\to \mathscr{G}\to 0,\]
where $\mathscr{E}$ is locally free and $\mathscr{G}$ is torsion-free.
\end{prop}

\begin{cor}
Under the above circumstances, the dual of a coherent sheaf is always reflexive.
\end{cor}

\begin{prop}[c.f.:\cite{Hart2}*{Theorem 1.9}]\label{refS2}
Let $p:X\to T$ be a morphism of noetherian schemes with normal integral fibers, and $\mathscr{F}$ a coherent sheaf on $X$. 
Then, if $\mathscr{F}$ is reflexive, it has relative property $S_2$ with respect to $p$.
\end{prop}
\begin{proof}[Proof.]
The statement being local, we can assume $X$ is quasi-projective.
Given a reflexive sheaf $\mathscr{F}$, we take an exact sequence
\[ 0\to \mathscr{F}\to \mathscr{L}\to \mathscr{G}\to 0\]
with $\mathscr{L}$ locally free and $\mathscr{G}$ torsion-free.
Since $p$ have normal fibers, $\O_X$ satisfies relative $S_2$ with respect to $p$, and so does $\mathscr{L}$, being locally free.
Let $x\in X$ be a point of relative dimension $\geq 2$ with respect to its fiber $X_{p(x)}$. 
Then $\mathrm{depth} \mathscr{L}_x \geq 2$ by $S_2$, and as $\mathscr{G}$ is torsion-free, $\mathrm{depth}\mathscr{G}_x \geq 1$.
This in turn implies $\mathrm{depth}\mathscr{F}_x \geq 2$.
\end{proof}

\subsection{Support of a sheaf, zeros of a section}\label{supzero}


\par\medskip Recall that, given a quasi-coherent sheaf $\mathscr{F}$ on a scheme $X$. we define the  \emph{support} of $\mathscr{F}$, $\mathrm{supp}(\mathscr{F})$ as the closed sub-scheme defined by the ideal sheaf given locally by
\[\mathcal{I}(\mathscr{F})_x:= \mathrm{Ann}(\mathscr{F}_x)\subset \O_{X,x}.\]

We have the following useful characterization of the support of a sheaf in terms of a universal property:

\begin{prop}
The support of a sheaf $\mathscr{F}$ represents the functor
\begin{eqnarray*}
\mathrm{S}_\mathscr{F}: Sch&\longrightarrow& Sets\\
T&\mapsto& \{ f\in \hom(T,X)\!: \mathrm{Ann}(f^*\mathscr{F})=0\subset \O_T \}.
\end{eqnarray*}•
\end{prop}
\begin{proof}[Proof.] 
A morphism $f:T\to X$ factorizes through $\mathrm{supp}(\mathscr{F})$ if and only if the map 
\[ f^\sharp: f^{-1}\O_X\to\O_T\]
factorizes through $f^{-1}(\O_X/\mathrm{Ann}(\mathscr{F}))$.
But this happens if and only if\break  $f^{-1}(\mathrm{Ann}(\mathscr{F}))=0$.
\par On the other hand we have the equality
\[
\mathrm{Ann}(f^*\mathscr{F})=\O_T\cdot f^{-1}(\mathrm{Ann}(\mathscr{F})),
\]
 indeed we may check this in every localization at any point $p\in T$, so if $t\in \mathrm{Ann}(f^*\mathscr{F})_p$ in particular $t$ annihilates every element of the form $m\otimes 1 \in \mathscr{F}_x$, so $t=f^{-1}(x) t'$ where $x\in \mathrm{Ann}(\mathscr{F})_{f(p)}$.
\par So $\mathrm{Ann}(f^*\mathscr{F})=0$ if and only if $ f^{-1}(\mathrm{Ann}(\mathscr{F}))=0$ and we are done.
\end{proof}

In other words we just proved that $\mathrm{supp}(\mathscr{F})$ is the universal scheme with the property that $f^*\mathscr{F}$ is not a torsion module. This simple observation will be very useful when discussing the scheme structure on the singular set of a foliation.

\par A special case of support of a sheaf is the scheme theoretic image of a morphism.
Remember that the \emph{scheme theoretic image} of a morphism $f:X\to Y$ is the sub-scheme $\mathrm{supp}(f_*\O_X)\subseteq Y$.

\par \medskip Now we turn our attention to sections and their zeros. 
So let $X$ be a scheme and $\mathscr{E}$ a locally free sheaf.
Having a global  section $s\in \Gamma(X,\mathscr{E})$ is the same as having a morphism (that, by abuse of notation,  we also call $s$)
\[s:\O_X \longrightarrow \mathscr{E}.\] 
Now, $s:\O_X\to \mathscr{E}$ defines a dual morphism
\[s^\vee: \mathscr{E}^\vee\longrightarrow \O_X^\vee=\O_X.\]

\begin{defs}\label{dzs}
We define the \emph{zero scheme} $Z(s)$ of the section $s$ as the closed sub-scheme of $X$ defined by the ideal sheaf  $\mathrm{Im}(s^\vee)\subseteq \O_X$.
\end{defs}

We'll apply this definition in the well behaved situation where $\O_X$ (and therefore $\mathscr{E}$) is torsion-free.

\begin{prop}\label{pzs}
Let $\mathscr{E}$ be a locally free sheaf on $X$ and $s\in \Gamma(X,\mathscr{E})$ a global section.
The scheme $Z(s)$ represents the functor 
\begin{eqnarray*}
\mathrm{Z}_s: Sch&\longrightarrow& Sets\\
T&\mapsto& \{ f\in \hom(T,X)\!: s\otimes 1=0\in \Gamma(T,f^*\mathscr{E})\}.
\end{eqnarray*}•
\end{prop}
\begin{proof}[Proof.]
A morphism $f:T\to X$ factorizes through $Z(s)$ if and only if the map
\[ (\mathscr{E}^\vee)\otimes \O_T\xrightarrow{s^\vee\otimes 1} \O_T\]
is identically $0$.
Beign locally free we have 
\[\mathscr{E}^\vee\otimes \O_T =\mathscr{H}om(\mathscr{E}, \O_X)\otimes \O_T\cong \mathscr{H}om(\mathscr{E}\otimes \O_T, \O_T).\]
So then we have 
\[ (f^*\mathscr{E})^\vee\xrightarrow{s^\vee\otimes 1} \O_T\]
is identically $0$, as $f^* \mathscr{E}$ is locally free over $T$, this means\break  $s\otimes 1=0\in \Gamma(T,f^*\mathscr{E})$.
\end{proof}

\subsection{A criterion for flatness}

For lack of a better reference we provide here a criterion that will become handy when dealing with both reduced and non-reduced base schemes over an algebraically closed field.

\begin{lema}\label{lile}
Let $A$ be a ring of finite type over an algebraically closed field $k$, $\mathcal{M}$ a maximal ideal in $A$, and $f\in \mathcal{M}^n\setminus\mathcal{M}^{n+1}$. Then there is a morphism $\psi: A\to k[T]/(T^{n+1})$ such that $\psi^{-1}((T))=\mathcal{M}$ and $\psi(f)\neq 0$.
\end{lema}
\begin{proof}[Proof.]
Set a presentation $A\cong k[y_1,\dots, y_r]/I$. By the Nullstelensatz we can assume $\mathcal{M}=(x_1,\dots,x_r)$, where $x_i$ is the class of $y_i\mathrm{mod} I$. Write the class of $f$ in $\mathcal{M}^n/\mathcal{M}^{n+1}$ as
\[ 
\-{f}=\sum_{|\alpha|=n}a_\alpha \-{x}^\alpha\in \mathcal{M}^n/\mathcal{M}^{n+1},
\]
where $\alpha=(\alpha_1,\dots,\alpha_r)$ and $\-{x}^\alpha=(\-{x_1}^{\alpha_1},\dots,\-{x_r}^{\alpha_r})$.
\par As $f \notin \mathcal{M}^{n+1}$, the polynomial $q(y_1,\dots, y_r):=\sum_{|\alpha|=n}a_\alpha y^\alpha$ is not in $I$. 
Now, $k$ being algebraically closed there is an $r$-tuple $(\lambda_1,\dots,\lambda_r)\in k^r$ such that $p(\lambda_1,...,\lambda_r)=0$ for every $p\in I$ and $q(\lambda_1,...,\lambda_r)\neq 0$.
\par Finally we can define $\psi:A\to k[T]/(T^{n+1})$ as follows:
\[
\psi(x_i)=\lambda_i T.
\]
The morphism is well defined because $p(\lambda_1,...,\lambda_r)=0$ for every $p\in I$, moreover $\psi^{-1}(T)=\mathcal{M}$, and $\psi(f)=q(\lambda_1,...,\lambda_r)T^n\neq 0$.
\end{proof}

\begin{prop}\label{rebusc}
 Let $f:X\to Y$ a projective morphism between schemes of finite type over an algebraically closed field, $\mathscr{F}$ a coherent sheaf over $X$, $x\in X$ a point, and $y=f(x)$. Then $\mathscr{F}_x$ is $f$-flat if and only if the following conditions hold:
\begin{enumerate}
	\item For every discrete valuation ring $A'$ and every morphism $\O_{Y,y}\to A'$ the following holds:\newline
	Taking the pull-back diagram 
	\[ 
	\xymatrix{
	&X' \ar[r] \ar[d]_{f'} &X \ar[d]^{f} \\
	&Y'=\spec(A') \ar[r] &Y
	}
	\]
	
	the $\O_{X'}$-module $\mathscr{F}'=\mathscr{F}\otimes \O_{Y'}$ is $f'$-flat at every point $x'\in X'$ lying over $x$.
	
	\item For every $n\in\N$ and every morphism $\O_{Y,y}\to k[T]/(T^{n+1})$, if we take the diagram analogous to the one above (with  $k[T]/(T^{n+1})$ instead of $A'$)  then the $\O_{X'}$-module $\mathscr{F}'=\mathscr{F}\otimes \O_{Y'}$ is $f'$-flat at every point $x'\in X'$ lying over $x$. 
\end{enumerate}
\end{prop}

\begin{proof}[Proof.] Clearly conditions 1 and 2 are necessary. Suppose then that 1 and 2 are satisfied.
\par Take the flattening stratification (see  \cite{FGex}*{Section 5.4.2}) of $Y$  with respect to  $\mathscr{F}$, $Y= \coprod_P Y_P$.
 As condition 1 is satisfied for $\mathscr{F}$ over $Y$, so is satisfied for $\iota^*\mathscr{F}$ over $Y_\mathrm{red}$, where $\iota:Y_\mathrm{red}\to Y$ is the closed immersion of the reduced structure. 
Then, by the valuative criterion for flatness of \cite{EGA4}*{11.8}, $\iota^*\mathscr{F}$ is flat over $Y_\mathrm{red}$, so by the universal property of the flattening stratification there is a factorization
\[
\xymatrix{ &Y_\mathrm{red} \ar[r]^\iota \ar[d] &Y\\
&\coprod_P Y_P \ar[ur] &
}
.\]
As $Y_\mathrm{red}$ and $Y$ share the same underlying topological set, the above factorization is telling us that the flattening factorization consist on a single stratum $Y_P$ and that $Y_\mathrm{red}\rightarrow Y_P$ is a closed immersion.
\par  Assume, by way of contradiction, $Y_P\subsetneq Y$, then there is an affine open sub-scheme $U\subseteq Y$ such that $V=Y_P\cap U\neq U$.
 Now take the coordinate rings $k[U]$ and $k[V]$ and the morphism between them induced by the inclusion $\phi:k[U]\epi k[V]$.
 Let's take $f\in k[U]$ such that $\phi(f)=0$. 
By \cref{lile} there exists, for some $n\in\N$, a morphism $\psi:k[U]\to k[T]/(T^{n+1})$ such that $\psi(f)\neq 0$, so $\psi$ doesn't factorize through $\phi$.
\par On the other hand, let $Z=\spec(k[T]/(T^{n+1}))$ and  $g:Z\to Y$ be the morphism  induced by $\psi$, as condition 2 is satisfied, the pull-back  $g^*\mathscr{F}$ is flat over $Z=\spec(k[T]/(T^{n+1}))$.
 So, by the universal property of the flattening stratification, $g$ factorizes as
\[
\xymatrix{ &Z\ar[r]^g \ar[d] &Y\\
& Y_P \ar[ur] &
}
,\]
contradicting the statement of the above paragraph, thus proving the proposition.
\end{proof}

Note that the hypothesis of this property on $X$ and $Y$ (aside from reducedness) are quite stronger than the ones of the original theorem of Grothendieck (the valuative criterion for flatness in \cite{EGA4}), such is the price we have paid to allow a criterion for possibly non-reduced schemes. 
The price paid is OK with us anyway, considering that we will work mostly with schemes of finite type over $\C$.
\newline
\par Next we provide a criterion for a $k[T]/(T^{n+1})$-module to be flat.

\begin{prop}\label{artcrit}
Let $A=k[T]/(T^{n+1})$ and $M$ an $A$-module. Then $M$ is flat if and only if for every $m\in M$ such that $T^n\cdot m=0$ there exist $m'\in M$ such that $m=T\cdot m'$.
\end{prop}

\begin{proof}[Proof.] 
Flatness of $M$ is equivalent to the injectivity of the map $M\otimes I\to M$ for every ideal $I\subset A$ (see e.g.:\cite[IV.1]{SGA1}). 
In this case there are finitely many ideals:
\[\mathcal{M}=(T),\ \mathcal{M}^2,\dots,\ \mathcal{M}^n.\]
If $M$ is flat is easy to see the second condition in our statement hold.
\par Suppose that  for every $m\in M$ such that $T^n\cdot m=0$ there exist $m'\in M$ such that $m=T\cdot m'$. Let $a\in M\otimes \mathcal{M}^{n-i}$ be in the kernel of $M\otimes \mathcal{M}^{n-i}\to M$. 
When $i=0$, we have $a=m\otimes T^n$, and $m$ is such that $T^n\cdot m=0$ so, by hypothesis, $m=T\cdot m'$ and then $m\otimes T^n= m'\otimes T^{n+1}=0$.
\par When $i>0$, we have $a=\sum_{j=n-i}^i m_j\otimes T^j$, so  $T^{i}\cdot a=m_{n-i}\otimes T^n\in M\otimes \mathcal{M}^n$.
By hypothesis,  $m_{n-i}=T\cdot m'$. So $a\in M\otimes \mathcal{M}^{n-i+1}$ and we are done by induction.
\end{proof}

The following will be useful in the study of foliations of codimension greater than $1$.

\begin{prop}\label{propstratfl}
Let $p:X\to S$ a projective morphism between schemes of finite type over an algebraically closed field $k$, $f: S\to Y$ another morphism, with $Y$ of finite type over $k$, and $\mathscr{F}$ a coherent sheaf over $X$.
Take a stratification $\coprod_i S_i\subseteq S$ of $S$ such that  $\mathscr{F}|_{S_i}:=\mathscr{F}\otimes_S \O_{S_i} $ is flat for all $i$.
If the composition $\coprod_i S_i\hookrightarrow S\xrightarrow{f} Y$ is a flat morphism, then $\mathscr{F}$ is flat over $Y$.

\end{prop}

\begin{proof}[Proof.]
Invoking \cref{rebusc} we can, after applying base change, reduce to the case where $Y$ is either the spectrum of a DVR or $Y=\spec(k[T]/(T^{n+1}))$.
\par {\bf (i) Case $Y=\spec(A)$ with $A$ a DVR.} 
Suppose there is, for some point $x\in X$ a section $s\in \mathscr{F}_x$ that is of torsion over $A$. 
Consider $Z=\mathrm{supp}_S(s)\subseteq S$ the support of $s$ over $S$, that is the support of $s$ as an element of $\mathscr{F}_x$ considered as an $\O_{S, p(x)}$-module.
Now take any stratum $S_i$ and suppose $Z\cap S_i\neq \emptyset$.
Then there is a section of the pullback $\mathscr{F}_{S_i}$ that is of torsion over $A$. But $\mathscr{F}_{S_i}$ is flat over $S_i$ which is in turn flat over $A$, so $\mathscr{F}_{S_i}$ is flat and $Z\cap S_i$ must be empty for every stratum $S_i$, i.e.: $s=0$.
\par {\bf (ii) Case  $Y=\spec(k[T]/(T^{n+1}))$.} One can essentially repeat the argument above, now taking the section $s$ to be such that $T^n s=0$ but $s \notin T\cdot\mathscr{F}_x$. 
\end{proof}

\begin{cor}
Take the flattening stratification $\coprod_P S_P\subseteq S$, of $S$ with respect to $\mathscr{F}$. 
If the composition $\coprod_P S_P\hookrightarrow S\xrightarrow{f} Y$ is a flat morphism, then $\mathscr{F}$ is flat over $Y$.
\end{cor}

\section{Families of sub-sheaves and their dual families}\label{duales}

\begin{defs}
Given a short exact sequence of sheaves 
\[
0\to \mathscr{G} \xrightarrow{\iota} \mathscr{F} \to \mathscr{H}\to 0,
\]
we apply to it the left-exact contravariant functor $\mathscr{F}\mapsto \mathscr{F}^\vee:=\mathcal{H}om_X(\mathscr{F},\O_X)$ to obtain exact sequences:
\begin{align}
0\to \mathscr{H}^\vee \to &\mathscr{F}^\vee  \to \mathrm{Im}(\iota^\vee)\to 0, \label{dual}\\
0\to \mathrm{Im}(\iota^\vee) \to  \mathscr{G}^\vee \to   &\mathcal{E}xt_X^1(\mathscr{H},\O_X)\to \mathcal{E}xt_X^1(\mathscr{F},\O_X).\notag
\end{align}
 We say that the exact sequence \eqref{dual}, is the \emph{dual exact sequence} of $0\to \mathscr{G} \xrightarrow{\iota} \mathscr{F} \to \mathscr{H}\to 0$. 
\end{defs}

\begin{lema}\label{reflex}
Let $0\to N  \xrightarrow{\iota} T\xrightarrow{\pi}  M \to 0$ be a short exact sequence of $R$-modules such that $T$ is reflexive and $M$ is torsion free. Then  $\mathrm{Im}(\iota^\vee)^\vee=N$ and $M=\mathrm{Im}(\pi^{\vee\vee})$.
\end{lema}

\begin{proof}[Proof.]First we take the duals in the short exact sequence to get a sequence
\[ 0\to\hom_R(M,R)\xrightarrow{\pi^\vee}\hom_R(T,R)\xrightarrow{\iota^\vee} \mathrm{Im}(\iota^\vee)\to 0 \]
Then we take duals one more time and, given that $T$ is reflexive and that $M$ is torsion-free, we get the diagram
\[
\xymatrix{
	&0 \ar[r]  &\mbox{Im}(\iota^\vee)^\vee \ar[r] &T^{\vee\vee}\ar[r]^{\pi^{\vee\vee}} \ar@{=}[d] &M^{\vee\vee} \ar[r]  &\mathrm{ext}^1_R( \mbox{Im}(\iota^\vee), R)  \\
	&0\ar[r] &N \ar[u] \ar[r] &T \ar[r] &M\ar@{^(->}[u] \ar[r] &0,}
\]
whose rows are exact.
\par Chasing arrows we readily see that the leftmost vertical arrow must be an isomorphism. Indeed, since the monomorphism $N\to T^{\vee\vee}$ factorizes as 
\[ N\to \mbox{Im}(\iota^\vee)^\vee \to T^{\vee\vee},\]
the second arrow being a monomorphism, so must $N\to\mbox{Im}(\iota^\vee)^\vee$ be. 
On the other hand, given $a\in \mbox{Im}(\iota^\vee)^\vee$, we can regard it, via the inclusion, as an element in $T^{\vee\vee}=T$, so we can compute $\pi(a)$. 
As the canonical map $\theta:	M\to M^{\vee\vee}$ is an inclusion we have that, $\theta\circ\pi(a)=\pi^{\vee\vee}(a)=0$, then $\pi(a)=0$, so $a\in N$. From this we have  $N\cong \mbox{Im}(\iota^\vee)^\vee$, wich implies $M=\mathrm{Im}(\pi^{\vee\vee})$.
\end{proof}

\subsection{Exterior Powers}\label{productoexterior}
When dealing with foliations of codimension/dimension greater than $1$ is usually convenient to work with $p$-forms. 
We'll need then to compare sub-sheaves $I\subset \Omega^1_X$ with their exterior powers $\wedge^pI\subset \Omega^p_X$. 
In order to do that we include the following statements, valid in a wider context.
\newline
We'll concentrate on flat modules and their exterior powers. 
This will be important when dealing with flat families of Pfaff systems of codimension higher than $1$ (see \cref{p-formas}).

\begin{lema}\label{lex2}
Let $A$ be a ring containing the field $\Q$ of rational numbers, and let $M$ be a flat $A$-module.
 Then, for every $p$, $\wedge^p M$ is also flat.
\end{lema}
\begin{proof}[Proof.]
If tensoring with $M$ is an exact functor, so is its iterate \break $-\otimes M\otimes\dots\otimes M$.
So $M^{\otimes p}$ is flat. 
As $A$ contains $\Q$, there is an anti-symmetrization operator 
\[ M^{\otimes p}\to \wedge^pM\]
which is a retraction of the canonical inclusion $\wedge^pM\subset M^{\otimes p}$.
This makes $\wedge^pM$ a direct summand of $M^{\otimes p}$, set $M^{\otimes p}=\wedge^pM \oplus R$ for some module $R$.
As the tensor power distributes direct sums (i.e.: ${(\wedge^pM\oplus R)\otimes N\cong}{ ( \wedge^pM\otimes N)\oplus (R\otimes N})$), so does their derived functors.
In particular we have, for every module $N$,
\[ 0=\mathrm{Tor}_1(M^{\otimes p},N)=\mathrm{Tor}_1(\wedge^pM,N)\oplus \mathrm{Tor}_1(R,N).\]
So $\wedge^p M$ is flat.
\end{proof}

Finally, we draw some conclusions regarding flat quotient. 
When dealing with Pfaff systems, we'll be interested in short exact sequence of the form 
\[ 0\to\wedge^p I\to \Omega^p_X\to \mathcal{G}\to 0, \]
arising from short exact sequences of flat modules
\[0\to I\to \Omega^1_X\to \Omega\to 0.\]
Note that, in general $\mathcal{G}\neq \wedge^p \Omega$.
Nevertheless, we can state:

\begin{prop}
Let $A$ be a ring containing $\Q$. 
Given an exact sequence
\[0\to M\to P\to N\to 0\]
of flat $A$-modules, we have an associated exact sequence 
\[0\to \wedge^pM\to \wedge^p P\to Q\to 0.\]
Then $Q$ is also flat.
\end{prop}
\begin{proof}[Proof.] 
 Q inherit a filtration from $\wedge^pP$:
\[ Q=  \wedge^p P/\wedge^pM=\-{F}^0\supseteq \-{F}^1\supseteq\dots \supseteq \-{F}^p=0,\]
with quotients  
\[F^i/F^{i+1}\cong \wedge^iM\otimes \wedge^{p-i}N.\]
Then $Q$ have a filtration all of whose quotients are flat, so $Q$ itself is flat.
\end{proof}

\section{Families of distributions and Pfaff systems}\label{distypfaff}

We will consider subsheaves of the relative tangent sheaf $T_SX$ and the relative differentials $\Omega^{1}_{X|S}$.
\begin{defs}
A \emph{family of distributions} is a short exact sequence
\[
 0\to T\F\to T_S X\to N_\F\to 0.
\]
The family is called flat if $N_\F$ is flat over the base $S$.\newline
A family of distributions is called \emph{involutive} if it's closed under Lie bracket operation, that is, if for every pair of local sections $X$, $Y\in T\F(V)$, we have $[X,Y]\in T\F(V)$ where $[-,-]$ is the Lie bracket in $T_SX(V)$.\newline
\par Likewise, a \emph{family of Pfaff systems} is just a short exact sequence
\[
 0\to I(\F)\to \Omega^1_{X|S}\to  \Omega^1_{\F}\to 0.
\]
It's called flat if $ \Omega^1_{\F}$ is flat.\newline
We will say that a family of Pfaff systems is \emph{integrable} if 
$d(I(\F))\wedge \bigwedge^rI(\F)=0\subset \Omega^{r+2}_{X|S}$;
where $d:\Omega^j_{X|S}\to \Omega^{j+1}_{X|S}$ is the \underline{relative}  de Rham differential, and $r$ is the generic rank of the sheaf $\Omega^1_{\F}$.
\end{defs}
\begin{rmk}
Observe that the relative differential $d:\Omega^j_{X|S}\to \Omega^{j+1}_{X|S}$ is not an $\O_X$-linear morphism.
 It is, however, $f^{-1}\O_S$-linear, so the sheaf $d(I(\F))\wedge \bigwedge^rI(\F)$, whose annihilation encodes the integrability of the Pfaff system, is actually a sheaf of $f^{-1}\O_S$-modules.
\end{rmk}

In particular the dual to a family of distributions is a family of Pfaff systems and vice-versa.

\begin{rmk}
The dual of an \underline{involutive} family of distributions is an \underline{integrable} family of Pfaff systems. 
Reciprocally, the dual of an integrable family of Pfaff systems is a family of involutive distributions. 
This is just a consequence of the Cartan-Eilenberg formula for the de Rham differential of a $1$-form applied to vector fields
\[d\omega(X,Y)=X(\omega(Y))-Y(\omega(X))-\omega([X,Y]).\]
Indeed, as involutiveness and integrability can be checked locally over sections, we can proceed as in \cite[Prop. 2.30]{warner}.
\end{rmk}

\begin{defs}
The \emph{dimension} of a family of distribution is the generic rank of $T\F$. Likewise, the \emph{dimension} of a family of Pfaff systems is the generic rank of $\Omega^1_\F$.
\end{defs}

If $p: X\to S$ is moreover projective, $S$ is connected, and the family is flat, so $T\F$ is a flat sheaf over $S$. 
Then for every $s\in S$ the Hilbert polynomial of $T\F_s$ is the same, and so is its generic rank (being encoded in the principal coeficient of the polynomial). The same occurs with families of Pfaff systems. 

\begin{rmk}\label{p-formas}
Frequently, in the study of foliations of codimension higher than $1$, is more convenient and better adapted to calculations  to work with an alternative description of foliations. 
Namely, one can define a codimension $q$ foliation on a variety $X$ as in \cite{Med}, with a global section $\omega$ of $\Omega^q_X\otimes \mathcal{L}$ such that:
\begin{itemize}
	\item $\omega$ is locally decomposable, i.e.: there is, for all $x\in X$ an open set such that 
	\[ \omega=\eta_1\wedge\dots\wedge\eta_q,\]
	with $\eta_i\in \Omega^1_X$.
	\item $\omega$ is integrable, i.e.: $\omega\wedge d\eta_i=0,\ 1\leq i\leq q$. 
\end{itemize}•
With this setting, studying flat families of codimension $q$ foliations (meaning here families of integrable Pfaff systems) as in \cite{split} and \cite{rational}, parametrized by a scheme $S$, amounts to studying  short exact sequences of flat sheaves:
\[
 0\to \mathcal{L}^{-1}\to \Omega^q_{X|S}\to \mathcal{G}\to 0,
\]
that are locally decomposable and integrable. 
By the results of \cref{productoexterior} a flat family of codimension $q$ Pfaff systems given as a sub-sheaf of $\Omega^1_{X|S}$ give rise to a flat family in the above sense.
\end{rmk}

\section{Universal families}\label{univ}

Now lets take a non-singular projective scheme  $X$, a polynomial $P\in \Q[t]$,  and consider the following functor 

\begin{eqnarray*}
\mathfrak{Inv}^P(X):Sch&\longrightarrow& Sets\\
S&\mapsto&\left\{ \parbox{8cm}{flat families ${0\to T\F \to T_S(X\times S)\to N_\F\to 0}$ of involutive distributions  such that $N_\F$ have Hilbert polynomial $P(t)$.} \right\}.
\end{eqnarray*}•

Say $p:X\times S\to X$ is the projection, so $T_S(X\times S)= p^*TX$.
Clearly one have $\mathfrak{Inv}^P(X)$ is a sub-functor of $\mathfrak{Quot}^P(X,TX)$.
  We are going to show that $\mathfrak{Inv}^P(X)$ is actually a \emph{closed} sub-functor of $\mathfrak{Quot}^P(X,TX)$ and therefore also representable.

So take the smooth morphism given by the projection  
\[ p_1:\mathrm{Quot}_P(X,TX)\times X\to \mathrm{Quot}_P(X,TX).\]
 Here we are taking as base scheme $S=\mathrm{Quot}_P(X,TX)$, then on the total space $S\times X=\mathrm{Quot}_P(X,TX)\times X$ we have the natural short exact sequence
\[
0\to \mathscr{F}\to p_2^*TX=T_S (S\times X)\to \mathscr{Q}\to 0.
\]

Now we consider the push-forward of this sheaves by $p_1$, as $X$ is proper, this push-forwards are coherent sheaves over $S$. 
In particular we have  maps of coherent sheaves over $\mathrm{Quot}^P(X,TX)$
\[ p_{1*}\mathscr{F}\otimes_S p_{1*}\mathscr{F} \xrightarrow{[-,-]} p_{1*}T_S(S\times X) \to p_{1*}\mathscr{Q} \]
induced by the maps over $S\times X$.
Note that while the Lie bracket on $T_S(S\times X)$ is  only $p_1^{-1}\O_S$-linear, the  map induced on the push-forwards is $\O_S$-linear, so is a morphism of coherent sheaves.
We then also have for any $m,n\in\Z$ the twisted morphisms
\[ 
p_{1*}\mathscr{F}(m)\otimes_S p_{1*}\mathscr{F}(n) \xrightarrow{[-,-]} p_{1*}T_S(S\times X)(m+n) \to p_{1*}\mathscr{Q}(m+n).\]

Note also that, as $p_1$ is a projective morphism, then there exist an $n\in \Z$ such that for any $m\geq n$ the natural sheaves morphism  over $S\times X$, $p_1^*p_{1*}(\mathscr{F})(m)\to \mathscr{F}(m)$ is an epimorphism. So if for some $f:Z\to S$ and some $m\geq n$ one have that the composition 
\[
 f^*p_{1*}\mathscr{F}(m)\otimes_Z f^* p_{1*}\mathscr{F}(m) \xrightarrow{[-,-]}f^* p_{1*}T_S(S\times X)(2m) \to f^* p_{1*}\mathscr{Q}(2m)
\]
is zero, then the map 
\[ 
(f\times id)^*\mathscr{F}(m)\otimes_{\pi_1^{-1} \O_Z} (f\times id)^*\mathscr{F}(m) \xrightarrow{[-,-]}  T_Z(Z\times X)(2m) \to (f\times id)^*\mathscr{Q}(2m) 
\]
is zero as well, here $\pi_1:Z\times X\to Z$ is the projection, which is by the way the pull-back of $p_1$.

Now to conclude the representability of $\mathfrak{Inv}^P(X)$ we need one important lemma.

\begin{lema}
Let $S$ be a noetherian scheme, $p:X\to S$ a projective morphism and $\mathscr{F}$ a coherent sheaf on $X$. 
Then $\mathscr{F}$ is flat over $S$ if and only if there exist some integer $N$ such that for all $m\geq N$ the push-forwards $p_*\mathscr{F}(m)$ are locally free.
\end{lema}

\begin{proof}[Proof.]
This is \cite{FGex}*{Lemma 5.5}.
\end{proof}

We can then take $m\in Z$ big enough so $p_{1*}\mathscr{F}(m)$ and $p_{1*}\mathscr{Q}(2m)$ are locally free and the morphism 
 $p_1^*p_{1*}(\mathscr{F})(m)\to \mathscr{F}(m)$ is epimorphism.
Then we can regard the composition 
\[ 
p_{1*}\mathscr{F}(m)\otimes_S p_{1*}\mathscr{F}(m) \xrightarrow{[-,-]} p_{1*}T_S(S\times X)(2m) \to p_{1*}\mathscr{Q}(2m)
\]
as a global section $\sigma$ of the locally free sheaf $\mathcal{H}om_S( p_{1*}\mathscr{F}(m)\otimes_S p_{1*}\mathscr{F}(m),  \mathscr{Q}(2m))$.
We can then make the following definition.

\begin{defs}
We define the scheme $\mathrm{Inv}^P(X)$ to be the zero scheme $Z(\sigma)$ (cf.: \cref{dzs}) of the section $\sigma$ defined above.
\end{defs} 

A direct application of \cref{pzs} to this definition together with the discussion so far immediately gives us the following.

\begin{prop}\label{proprep1}
The subscheme $\mathrm{Inv}^P(X)\subseteq \mathrm{Quot}^P(X,TX)$ represents the functor $\mathfrak{Inv}^P(X)$.
\end{prop}

\par \bigskip Similarly we can consider the sub-functor $\mathfrak{iPf}^P(X)$ of $\mathfrak{Quot}^P(X,\Omega^1_X)$.

\begin{eqnarray*}
\mathfrak{iPf}^P(X):Sch&\longrightarrow& Sets\\
S&\mapsto&\left\{ \parbox{8cm}{flat families ${0\to I(\F) \to \Omega^1{X|S}\to \Omega^1_\F\to 0}$ of integrable Pfaff systems  such that $\Omega^1_\F$ have Hilbert polynomial $P(t)$.} \right\}.
\end{eqnarray*}•

Then as before we take $S=\mathrm{Quot}^P(X,\Omega^1_X)$ and consider the map
\[ p_{1*}(d(\mathscr{I})\wedge \bigwedge^r\mathscr{I})(m) \longrightarrow p_{1*}\Omega^{r+2}_{S\times X|S}(m).\]
 
Which is, for large enough $m$ a morphism between locally free sheaves on $S$. 

\begin{defs}
We define the scheme $\mathrm{iPf}^P(X)$ to be the zero scheme of the above morphism, viewed as a global section of the locally free
sheaf $\mathcal{H}om( p_{1*}(d(\mathscr{I})\wedge \bigwedge^r\mathscr{I})(m),  p_{1*}\Omega^{r+2}_{S\times X|S}(m))$.
\end{defs}

And then by \cref{pzs} we have representability.

 \begin{prop}\label{proprep2}
The subscheme  $\mathrm{iPf}^P(X)\subset \mathrm{Quot}^P(X,\Omega^1_X)$ represents the functor  $\mathfrak{iPf}^P(X)$.
\end{prop}

\section{Duality}\label{duality}

\begin{defs}\label{singf}
The \emph{singular locus} of a family of distributions  \break$ 0\to T\F\to T_S X\to N_\F\to 0$ is the (scheme theoretic) support of $\mathcal{E}xt^1_X(N_\F,\O_X)$. Intuitively, its points are the points where $N_\F$ fails to be a fiber bundle.\par
Similarly, for a family of Pfaff systems $ 0\to I(\F)\to \Omega^1_{X|S}\to  \Omega^1_{\F}\to 0$, its singular locus is $\mathrm{supp}( \mathcal{E}xt^1_X(\Omega^1_{\F},\O_X))$.
\end{defs}

\begin{rmk}
Call $i:T\F\to T_SX$ the inclusion.
We have an open non-empty set $U$ where, for every $x\in U$,  $\dim(Im(i\otimes k(x)))$ is maximal.
More precisely, $U$ is the open set where $Tor_1^X(N_\F, k(x))=0$, which is the maximal open set such that $N_\F|_U$ is locally free, and therefore so is $T\F$.
Then, when restricted to $U$, $T\F$ can be given locally as the subsheaf of $T_SX$ generated by $k$ linearly independent relative vector fields, i.e.: $T\F$ defines a family of non-singular foliations. 
In $U$, one have that $\mathcal{E}xt^1_X(N_\F, \O_X)=0$. 
Then, the underlying topological space of the singular locus of the family given by $T\F$ is the singular set of the foliation in a classical (topologial space) sense. 
\par The above discussion translates verbatim to families of Pfaff systems.
\end{rmk}

\begin{prop}
Let 
\[
 0\to I(\F)\to \Omega^1_{X|S}\to  \Omega^1_{\F}\to 0 \]
be a family of Pfaff systems such that $ \Omega^1_{\F}$ is torsion-free. Its singular locus and the singular locus of the dual family
\[
 0\to T\F\to T_S X\to N_\F\to 0 \]
are the same sub-scheme of $X$. We denote this sub-scheme by $\mathrm{sing}(\F)$
\end{prop}
\begin{proof}[Proof.]
We are going to show that the immersions $Y_1:=\mathrm{supp}( \mathcal{E}xt^1_X(\Omega^1_{\F},\O_X))\subseteq X$ and $Y_2:=\mathrm{supp}( \mathcal{E}xt^1_X(N_\F,\O_X))\subseteq X$ represent the same sub-functor of $\mathrm{Hom}(-,X)$, thus proving the proposition.\newline
\newline
First note that, if  $\mathcal{E}xt^1_X(N_\F,\O_X)=0$, then $ \mathcal{E}xt^1_X(\Omega^1_{\F},\O_X)=0$.\newline
Indeed, if  $\mathcal{E}xt^1_X(N_\F,\O_X)=0$, $N_\F$ is locally free and then so is $T\F$. Moreover, since $\Omega^1_{\F}$ is torsion free, we can dualize the short exact sequence $ 0\to T\F\to T_S X\to N_\F\to 0 $ and, by lemma \ref{reflex}, obtain the equality $\Omega^1_{\F}=T\F^\vee$.\newline
 So $\Omega^1_{\F}$ is locally free and $ \mathcal{E}xt^1_X(\Omega^1_{\F},\O_X)=0$.\newline
\newline
Now, given a quasi-coherent sheaf $\mathscr{G}$ of $X$, its support $\mathrm{supp}(\mathscr{G})\subseteq X$ represents the following sub-functor of $\mathrm{Hom}(-,X)$:
\[ T\longmapsto \{f:T\to X\quad \text{s.t.: $f^*\mathscr{G}$ is not a torsion sheaf}\}\subseteq \mathrm{Hom}(T,X).\]
So, let's take a morphism $f:T\to  Y_1\subseteq X$. 
\begin{enumerate}[(i)]	
	\item $f:T\to  Y_1$ is an immersion:
		 Suppose   $f^*\mathcal{E}xt^1_X(N_\F,\O_X)$ is a torsion sheaf. \break Then there's a point $t\in T$ such that
	\[ \mathcal{E}xt^1_X(N_\F,\O_X)\otimes k(t)=0.\]
 By Nakayama's lemma this implies that there's an open subset $U\subseteq X$ containing $t$ such that $ \mathcal{E}xt^1_X(N_\F,\O_X)|_U =0$. This in turn implies  $ \mathcal{E}xt^1_X(\Omega^1_{\F},\O_X)|_U =0$ contradicting the fact that $t\in T\subseteq Y_1$. Then $T\subseteq Y_2$.\newline
	Similarly one proves that if $T\subseteq Y_2$ then $T\subseteq Y_1$.

	\item General case: Taking the scheme theoretic image of $f$ we can reduce to the above case where $T$ is a sub-scheme of $X$.
\end{enumerate}
\end{proof}

\subsection{The codimension $1$ case}\label{codim1case}
We now treat the case of families of codimension $1$ foliations. From now on we'll suppose that $X\to S$ is a smooth morphism.
\begin{defs} A family of involutive distributions 
\[
 0\to T\F\to T_S X\to N_\F\to 0,
\]
 is of \emph{codimension} $1$ iff $N_\F$ is a sheaf of generic rank $1$.
\par Likewise a family of Pfaff systems 
\[
 0\to I(\F)\to \Omega^1_{X|S}\to  \Omega^1_{\F}\to 0,
\]
is of codimension $1$ if the sheaf $I(\F)$ have generic rank $1$.
\end{defs}
\begin{lema}
Let be  a family of codimension $1$ Pfaff systems 
\[
 0\to I(\F)\to \Omega^1_{X|S}\to  \Omega^1_{\F}\to 0,
\]
over an integral scheme $X$, such that $\Omega^1_{\F}$ is torsion-free.
Then $I(\F)$ is a line-bundle over $X$.
\end{lema}
\begin{proof}[Proof.] 
If  $\Omega^1_{\F}$ is torsion-free, by \cref{reflex} we have $I(\F)\cong N_\F^\vee$.
In particular $I(\F)$ is the dual of a sheaf, then is reflexive and observes property $S_2$.
Write $I=I(\F)$ and consider now the sheaf $I^\vee\otimes I$ together with the canonical morphism
\[ I^\vee\otimes I\to \O_X.\]
The generic rank of $I^\vee\otimes I$ is $1$. 
As $I$ is reflexive, $I^\vee\otimes I$ is self-dual.
So the canonical morphism above induces the dual morphism $\O_X\to I^\vee\otimes I$.
The composition
\[\O_X\to I^\vee\otimes I \to \O_X \]
must be invertible, otherwise the image of $I^\vee\otimes I$ in $\O_X$ would be a torsion sub-sheaf.
Then $I$ is an invertible sheaf.
\end{proof}

\begin{prop}\label{propdual1}
In the case of codimension $1$ Pfaff systems, if $\Omega^1_{\F}$ is torsion-free over $X$ and the inclusion $I(\F)\to \Omega^1_{X|S}$ is nowhere trivial on $S$ (meaning that $I(\F)\otimes \O_T\to \Omega^1_{X|S}\otimes \O_T$ is never the zero morphism for any $T\to S$) then the family is automatically flat.
\end{prop}
\begin{proof}[Proof.]
  Indeed, $\Omega^1_{\F}$ being torsion free implies that the rank-$1$ sheaf $I(\F)$ must be a line bundle. 
 Then if we take any morphism $f:T\to S$ and take pull-backs we'll have an exact sequence
\[0\to \mathrm{Tor}_1^S(\Omega^1_{\F}, T)\to f^*I(\F)\to f^*\Omega^1_{X|S}\to f^*\Omega^1\to 0.\]
Now, as $I(\F)$ is a line bundle, the cokernel $ f^*I(\F)/ \mathrm{Tor}_1^S(\Omega^1_{\F}, T)$ must be a torsion sheaf over $X_T$.
But, $X$ being smooth over $S$, the annihilator $f^*\Omega^1_{X|S}$ is of the form $p^*(J)$, with $J\subset\O_T$, so $ f^*I(\F)\to f^*\Omega^1_{X|S}$ must be the zero morphism when restricted to $\O_T/J$, contradicting the nowhere triviality assumption.
\end{proof}

\begin{rmk}
In the codimension $1$ case, we can calculate the $\sing(\F)$ by noting that $\mathcal{E}xt^1_X(\Omega^1_{\F}, \O_X)$ is the cokernel in the exact sequence
\[T_SX\to I(\F)^\vee\to \mathcal{E}xt^1_X(\Omega^1_{\F}, \O_X)\to 0.\]
We can then tensor the sequence by $I(\F)$ and obtain 
\[ T_S X\otimes I(\F)\to \O_X\to  \mathcal{E}xt^1_X(\Omega^1_{\F}, \O_X)\otimes I(\F)\to 0.\]
Now, $I(\F)$ being a line bundle, the support of $ \mathcal{E}xt^1_X(\Omega^1_{\F}, \O_X)$ and that of $ \mathcal{E}xt^1_X(\Omega^1_{\F}, \O_X)\otimes I(\F)$ is exactly the same.
Note then that, in the second exact sequence, the cokernel is the scheme theoretic zero locus of the twisted $1$-form given by
\[ \O_X\xrightarrow{\omega} \Omega^1_{X|S}\otimes I(\F)^\vee \] 
as defined in \cref{supzero}.
So, if we have a family of codimension $1$ Pfaff systems given locally by a twisted form 
\[ \omega=\sum_{i=1}^n f_i(x) dx_i\]
then $\sing(\F)$ is the scheme defined by the ideal $(f_1, \dots, f_n)$.
\end{rmk}

The above proposition and remark tell us that our definition of flat family for Pfaff systems of codimension $1$ is essentially the same as the one used in the now classical works of Lins-Neto, Cerveau, et. al.

\begin{teo}\label{teo}
Assume we have two families
\begin{align}
 0\to I(\F)\to &\Omega^1_{X|S}\to  \Omega^1_{\F}\to 0 \label{sec1}\\
 0\to T\F\to &T_S X\to N_\F\to 0, \label{sec2}
\end{align}
satisfying the following conditions:
\begin{itemize}
\item The families \ref{sec1} and \ref{sec2} are dual to each other.
\item  $N_\F$ is torsion free (or, equivalently, $\Omega^1_{\F}$ is torsion free).
\item They are codimension $1$ families.
\item $\mathrm{sing}(\F)$ is flat over $S$.\newline
\end{itemize}
Then \ref{sec1} is \underline{flat} if and only if \ref{sec2} is flat. 

\end{teo}
\begin{proof}[Proof.]
Let $\Sigma=\mathrm{sing}(\F)$.
\par Let's suppose first that the family
\[ 
 0\to I(\F)\to \Omega^1_{X|S}\to  \Omega^1_{\F}\to 0
\]
is flat. We have to prove that $N_\F$ is also flat.%
To do this we note that applying the functor $\mathcal{H}om_X(-,\O_X)$ to the family of distributions not only gives us the family of Pfaff systems but also the exact sequence
\[
 0 \to N_\F\to I(\F)^\vee\to \mathcal{E}xt^1_X(\Omega^1_{\F}, \O_X)\to 0.
\]
Being $\Omega^1_{\F}$ torsion-free, $I(\F)$ must be a line bundle, and so must $I(\F)^\vee$, let's call $I(\F)^\vee=\mathcal{L}$ to ease the notation. 
Now $\mathcal{L}$ have $N_\F$ as a sub-sheaf generically of rank $1$, so $N_\F=\mathcal{I}\cdot\mathcal{L}$ for some Ideal sheaf $\mathcal{I}$. 
Then $ \mathcal{E}xt^1_X(\Omega^1_{\F}, \O_X)\cong \mathcal{L}\otimes \O_X/\mathcal{I}$. 
As $\Sigma=\mathrm{supp}( \mathcal{E}xt^1_X(\Omega^1_{\F}, \O_X))$ one necessarily have $\mathcal{E}xt^1_X(\Omega^1_{\F}, \O_X)\cong \mathcal{L}_\Sigma$. Then $\mathcal{L}_\Sigma$, being a locally free sheaf over $\Sigma$ wich is flat over $S$, is itself flat over $S$.
Therefore, as $\mathcal{L}$ is also flat over $S$, flatness for $N_\F$ follows.
\par Let's suppose now that the family 
\[
0\to T\F\to T_S X\to N_\F\to 0
\]
is flat. We have to prove that $\Omega^1_{\F}$ is also flat.
 By the above proposition it's enough to show that the morphism $I(\F)\xrightarrow{\iota} \Omega^1_{X|S}$ is nowhere zero.
 Suppose there is $T\to S$ such that $\iota_T=0$.
 Take an open set $U\subset X$ where $\Omega^1_{\F}$ is locally free. 
In that open set we can apply base change with respect to the functor   $\mathcal{H}om_{X}(-,\O_X)$ (\cite{AK} or \cite{BO}) so, restricting everything to $U$ we have $(\iota_T)^\vee\cong (\iota^\vee)_T$.
 But, in $U$, $\iota^\vee$ is the morphism $T_SX\to N_\F$ and so it cannot become the zero morphism under any base change.
\end{proof}

\subsection{The arbitrary codimension case}\label{codimk}

To give an  analogous theorem to \ref{teo} in arbitrary codimension we'll have to deal with finer invariants than the singular locus of the foliation. 
In the scheme $X$ we'll consider a stratification naturally associated with ${\F}$.
This stratification have been already studied and described by Suwa in \cite{Suwa}.
To deal with flatness issues we have to provide a scheme structure to Suwa's stratification, this will be a particular case of flattening stratification.
Before going into that, we begin with some generalities.
Remember that we are working over a smooth morphism $X\to S$.

\begin{lema}
Let $X\to S$ be a smooth morphism, $\mathscr{F}$ a coherent sheaf on $X$ that is relatively $Z^{(2)}$-closed over $S$.
Then, for any $s\in S$, the sheaf $\mathscr{F}_s=\mathscr{F}\otimes k(s)$ is $Z^{(2)}$-closed over $X_s$.
\end{lema}
\begin{proof}[Proof.] 
We have to show that for every $U\subset X_s$ such that $\mathrm{codim}(X\setminus U)\geq 2$ the restriction 
\[\mathscr{F}_s\xrightarrow{\rho_U} \mathscr{F}_s|_U\]
is surjective.
As the formal completion $ \widehat{\O}_{X_s, x}$ of $\O_{X_s}$ with respect to any closed point $x$ is faithfully flat over $\O_{X_s}$ \cite{SGA1}*{IV.3.2},  we can check surjectivity of $\rho_U$ by looking at every formal completion.
As $X\to S$ is smooth, formally around a point $x$ we have $\O_X\cong \O_S\otimes_k k[z_1,\dots, z_d]$ so we can take an open subset $V\subseteq X$, to be $V=U\times S$. 
Then, with this choice of $V$, we have an epimorphism
\[  \^{\mathscr{F}}|_V\to \^{\mathscr{F}}_s|_U\to 0. \]
Then we have  a diagram with exact rows and columns 
\[
\xymatrix{ 
&\^{\mathscr{F}} \ar[r] \ar^{\rho_V}[d] &\^{\mathscr{F}_s} \ar^{\rho_U}[d] \ar[r] &0 \\
&\^{\mathscr{F}}|_V \ar[r] \ar[d] &\^{\mathscr{F}_s}|_U \ar[r] &0\\
&0
}.
\]
So $\rho_U$ must be an epimorphism as well.

\end{proof}

\begin{lema}\label{excdist}
Let $p:X\to S$ a smooth morphism.
Let be a family of distributions
\[0\to T\F\to T_S X\to N_\F\to 0.\]
If the codimension of  $\mathrm{sing}(\F)$ with respect to $X_{p(\mathrm{sing}(\F))}$ is greater than $2$ then, for every map $T\to S$, one have
\[\mathcal{H}om_X(T\F,\O_X)\otimes \O_T \cong \mathcal{H}om_{X_T}(T\F_T,\O_T) .\]
The analogous statement is true for $I(\F)^\vee$ in a flat family of Pfaff systems.
\end{lema}
\begin{proof}[Proof.]
By \cite{AK}*{Theorem 1.9}  we only have to prove that, for every closed point $s\in S$, the natural map
\[ \mathcal{H}om_X(T\F,\O_X)\otimes k(s) \to \mathcal{H}om_{X_s}(T\F\otimes k(s),\O_X \otimes k(s))\]
is surjective.
Being the dual of some sheaves, both $\mathcal{H}om_X(T\F,\O_X)$ and $\mathcal{H}om_{X_s}(T\F\otimes k(s),\O_X \otimes k(s))$ possess the relative property $S_2$ with respect to $p$ ( \cref{refS2}), and so are relatively $Z^{(2)}$-closed w.r.t. $p$,
and so is $\mathcal{H}om_X(T\F,\O_X)\otimes k(s)$ by the above lemma.
\par Let $U=X\setminus \sing{\F}$ and $j:U\hookrightarrow X$ the inclusion. 
 As $T\F|_U$ is locally free over $U$, so is $T\F^\vee|_U$. 
Then, in $U$, we have 
\[ \mathcal{E}xt^1(T\F|_U, \O_X|_U \otimes_S \mathcal{G})=0,\]
for every $\mathcal{G}\in Coh(S)$. 
Then from the exchange property for local Ext's \cite{AK}*{Theorem 1.9} we get surjectivity on
\[ \mathcal{H}om_X(T\F|_U,\O_X|_U)\otimes k(s) \to \mathcal{H}om_{X_s}(T\F|_U\otimes k(s),\O_X|_U \otimes k(s)).\]
But, as $\mathrm{codim}(\sing(\F))>1$ and both sheaves are $S_2$, then surjectivity holds in all of $X_s$.
\end{proof}

\begin{lema}
Given a flat family 
\[0\to T\F\to T_S X\to N_\F\to 0.\]
Such that the relative codimension $d_S(\mathrm{sing}(\F))$ of $\mathrm{sing}(\F)$ over $S$ verifies $d_S(\mathrm{sing}(\F))\geq 2$. 
Suppose further that the flattening stratification of $X$ over $T\F$ is flat over $S$ (c.f.: \cref{propstratfl}).
  Then $T\F^\vee$ is also a flat $\O_S$-module.
\par The analogous statement is true for $I(\F)^\vee$ in a flat family of Pfaff systems.
\end{lema}
\begin{proof}[Proof.]
The proof works exactly the same for distributions or Pfaff systems mutatis mutandi.
\par Take $\coprod_P X_P$ the flattening stratification of $X$ with respect to $T\F$.
The restriction $T\F_{X_P}$ (being coherent and flat over $X_P$) is locally free over $X_P$, then so is its dual $\mathcal{H}om_{X_P}(T\F_{X_P},\O_{X_P})$.
By \cref{excdist}, in each stratum $X_P$ we have the isomorphism
\[ \mathcal{H}om_{X_P}(T\F_{X_P},\O_{X_P}) \cong \mathcal{H}om_X(T\F,\O_X)\otimes \O_{X_P}=T\F^\vee \otimes \O_{X_P}.\]
So $T\F^\vee$ is flat when restricted to the filtration $\coprod_P X_P$, which is in turn flat over $S$.
Then, by \cite{FGex}*{Section 5.4.2}, $T\F^\vee$ is flat over $S$.
%
\end{proof}

\begin{defs}
For a family of distributions
consider the flattening stratification 
\[\coprod_{P(\F)} X_{P(\F)} \subseteq X\]
 of $X$ with respect to $T\F\oplus N_\F$.
We call this the \emph{rank stratification} of $X$ with respect to $T\F$.
\end{defs}

\begin{rmk}
Note that the flattening stratification of  $T\F\oplus N_\F$ is the (scheme theoretic) intersection of the flattening stratification of $T\F$ with that of $N_\F$. This is because $(T\F\oplus N_\F)\otimes \O_Y$ is flat if and only if both $T\F\otimes \O_Y$ and $  N_\F\otimes \O_Y$ are. 
\par  This tells us, in particular, that each stratum  is indexed by two natural numbers $r$ and $k$ such that
\[ x\in X_{r,k} \Longleftrightarrow \dim(T\F\otimes k(x))=r\ \text{and}\ \dim(N_\F\otimes k(x))=k. \]
\par In \cite{Suwa}, Suwa studied a related stratification associated to a foliation. 
Given a distribution  $D\subset TM$ on a complex manifold $M$, he defines the strata $M^{(l)}$ as
\[   M^{(l)}=\{ x\in M\ s.t.\!: D_x \subset T_xM\ \text{is a sub-space of dimension $l$ }\}.\]
Here $D$ is spanned point-wise by vector fields $v_1,\dots,v_r$, and $D_x=<v_i(x)>$.
Clearly if $D$ is of generic rank $r$ the open stratum is $M^{(r)}$. 
\par Note that, in the setting of distribution as sub-sheafs $i: T\F \hookrightarrow TX$ of the tangent sheaf of a variety, the vector space $T_x\F$ is actually the image of the map 
\[T\F \otimes k(x) \xrightarrow{i\otimes k(x)} TX\otimes k(x),\]
whose kernel is $\mathrm{Tor}^X_1(N_\F, k(x))$.
Moreover we have the exact sequence
\[ 0\to T_x\F=\mathrm{Im}(i\otimes k(x))\to TX\otimes k(x) \to N_\F\otimes k(x)\to 0.\]
In particular, in a variety $X$ of dimension $n$, if $\dim( T_x\F)=l$ then $\dim(N_\F\otimes k(x))=n-l$. 
So what we call rank stratification of $X$ is actually a refinement of the stratification studied in \cite{Suwa}.
\end{rmk}

Our main motivation for defining this refinement of the stratification of \cite{Suwa} is the following result.

\begin{teo}\label{propdual2}
Assume we have dual families
\begin{align}
0\to T\F\to T_S X\to N_\F\to 0, \label{sec3}\\
0\to I(\F)\to \Omega^1_{X|S}\to \Omega^1_\F\to 0, \label{sec4}
\end{align}
parametrized by a scheme $S$ of finite type over an algebraically closed field, such that
\begin{itemize}
\item $N_\F$ is torsion free 
\item The relative codimension of $\mathrm{sing}(\F)$ over $S$ (that is $d_S(\mathrm{sing}(\F)$)) verifies $d_S(\mathrm{sing}(\F)\geq 2$. 
\item Each stratum $X_{r,k} $ of the rank stratification is flat over $S$.
\end{itemize}•
Then \ref{sec3} is flat over $S$ if and only if \ref{sec4} is.
Moreover,  for each point $s\in S$ we have
\[I(\F)_s = (N_{\F s})^\vee,\]
in other terms ``the dual family is the family of the duals".
\end{teo}
\begin{proof}[Proof] 
We prove one of the implications, the proof of the other is identical.
\par
Considering the exact sequence
\[ 0\to \Omega^1_\F \to T\F^\vee \to \mathcal{E}xt^1_X(N_\F,\O_X)\to 0.\] 
Is clear that to prove flatness of the dual family it's enough to show $\mathcal{E}xt^1_X(N_\F,\O_X)$ is flat over $S$.
\par Also, by \cref{excdist}, we have for every $s\in S$ the diagram with exact rows and columns,
\[
\xymatrix{
&\mathcal{H}om(I(\F),\O_X)\otimes k(s) \ar[r] \ar[d] &\mathcal{H}om(I(\F)_s,\O_{X_s}) \ar[r] \ar[d] &0\\
&\mathcal{E}xt_X^1(\Omega^1_\F, \O_X)\otimes k(s) \ar[r] \ar[d] &\mathcal{E}xt^1_X(\Omega^1_{\F}\otimes k(s), \O_{X_s}) \ar[d]\\
&0 &0 
}.
\]
So $\mathcal{E}xt_X^1(\Omega^1_\F, \O_X)\otimes k(s)\to \mathcal{E}xt^1_X(\Omega^1_{\F}\otimes k(s), \O_{X_s})$ is surjective for every $s\in S$ so by \cite{AK}*{Theorem 1.9} the exchange property is valid for $\mathcal{E}xt_X^1(\Omega^1_\F, \O_X)$.
 If  moreover $\mathcal{E}xt^1_X(N_\F,\O_X)$ is flat over $S$, then, again by  \cite{AK}*{Theorem 1.9},
\[
I(\F)_s=\mathcal{H}om_X(N_\F, \O_X)\otimes k(s) \cong \mathcal{H}om_X(N_{\F s}, \O_{X_s})=(N_{\F  s})^\vee .
  \]
By \cref{propstratfl} is enough to show the restriction of  $\mathcal{E}xt^1_X(N_\F,\O_X)$ to every rank stratum is flat over $S$.
So let $Y\subseteq X$  be a rank stratum, if we can show that   $\mathcal{E}xt^1_X(N_\F,\O_X)\otimes \O_Y$ is locally free then we're set.
By hypothesis, one have the isomorphism $\mathcal{H}om_X(T\F,\O_X)\otimes \O_Y \cong \mathcal{H}om_Y(T\F_Y,\O_Y)$.
So we can express  $\mathcal{E}xt^1_X(N_\F,\O_X)\otimes \O_Y$ as the cokernel in the $\O_Y$-modules exact sequence
\[ \mathcal{H}om_Y(TX_Y,\O_Y) \to \mathcal{H}om_Y(T\F_Y,\O_Y) \to \mathcal{E}xt^1_X(N_\F,\O_X)_Y\to 0.\]
So, localizing in a point $y\in Y$, we can realize the local $\O_{Y,y}$-module $\mathcal{E}xt^1_X(N_\F,\O_X)_{Y,y}$ as the set of maps $T\F_{Y,y}\to \O_{Y,y}$ modulo the ones that factorizes as
\[ \xymatrix{
&T\F_{Y,y} \ar@{-->}[r] \ar[d] &\O_{Y,y}\\
&TX_{Y,y} \ar[ru] }.
\]
To study  $\mathcal{E}xt^1_X(N_\F,\O_X)_{Y,y}$ this way, note that we have the following exact sequence.
\[0\to\mathrm{Tor}^X_1(N_\F, \O_{Y,y})\to T\F_{Y,y}\to TX_{Y,y}\to (N_\F)_{Y,y}\to 0.\] 
Which we split into two short exact sequences,
\begin{eqnarray}
0\to\mathcal{K}\to &TX_{Y,y}&\to (N_\F)_{Y,y}\to 0 \quad \text{and} \label{ace1} \\
0\to\mathrm{Tor}^X_1(N_\F, \O_{Y,y})\to &T\F_{Y,y}&\to \mathcal{K}\to 0. \label{ace2}
\end{eqnarray}•
Now, as $Y$ is a rank stratum, then $\mathcal{Q_Y}$ and $T\F_Y$ are flat over $Y$, and coherent, so they are locally free.
As a consequence, short exact sequence (\ref{ace1}) splits, so $TX_{Y,y}\cong (N_\F)_{Y,y}\oplus \mathcal{K}$.
So 
\[\mathcal{H}om_Y(TX_Y,\O_Y)_y\cong {\mathcal{H}om_Y(\mathcal{K},\O_{Y,y})\oplus \mathcal{H}om_Y((N_\F)_Y,\O_Y)_y}\] and we get $\mathcal{E}xt^1_X(N_\F,\O_X)_{Y,y}$ as the cokernel in 
\begin{equation}\label{ace3}
 \mathcal{H}om_Y(\mathcal{K},\O_{Y,y}) \to \mathcal{H}om_Y(T\F_Y,\O_Y)_y \to \mathcal{E}xt^1_X(N_\F,\O_X)_{Y,y}\to 0.
\end{equation}•

Being $(N_\F)_Y$ and $TX_Y$ locally free over $Y$, so is $\mathcal{K}$.
Then short exact sequence (\ref{ace2}) splits, so $T\F_{Y,y} \cong \mathrm{Tor}^X_1(N_\F, \O_{Y,y}) \oplus \mathcal{K}$.
Also, as $T\F_Y$ and $\mathcal{K}$ are locally free over $Y$, so is $\mathrm{Tor}^X_1(N_\F, \O_{Y})$ . 
 Sequence (\ref{ace3}) now reads
\[
 \mathcal{H}om_Y(\mathcal{K},\O_{Y,y}) \to \mathcal{H}om_Y(\mathrm{Tor}^X_1(N_\F, \O_{Y,y}),\O_Y)_y \oplus \mathcal{H}om_Y( \mathcal{K}, \O_{Y,y}) \to \mathcal{E}xt^1_X(N_\F,\O_X)_{Y,y}\to 0.
\]
So we have 
\[ \mathcal{H}om_Y(\mathrm{Tor}^X_1(N_\F, \O_{Y,y}),\O_Y)_y \cong \mathcal{E}xt^1_X(N_\F,\O_X)_{Y,y}.\]
Now, as $\mathrm{Tor}^X_1(N_\F, \O_{Y,y})$ is locally free over $Y$, so is its dual. 
In other words, we just proved $\mathcal{E}xt^1_X(N_\F,\O_X)_{Y}$ is locally free over $Y$, which settles the theorem. \end{proof}

Now by generic flatness we can conclude the following.

\begin{cor}\label{corobir}
Every irreducible component of the scheme $\mathrm{Inv}_P$ is birationally equivalent to an irreducible component of $\mathrm{iPf}_P$.
\end{cor}

\begin{rmk}
During the proof of \cref{propdual2} we have actually obtained this result:
\begin{prop}
$\mathcal{E}xt^1_X(N_\F,\O_X)$ is flat over the rank stratification.
\end{prop}
In particular, if $\coprod X_Q$ denotes the flattening stratification of $\mathcal{E}xt^1_X(N_\F,\O_X)$, there is a morphism 
\[\coprod_{P(\F)}X_{P(\F)}\to \coprod_Q X_Q.\]
Now, by the construction of flattening stratification, $\coprod X_Q$ consist of an open stratum $U$ such that  $\mathcal{E}xt^1_X(N_\F,\O_X)|_U=0$, and closed strata whose closure is $\mathrm{sing}(\F)$. 
So the morphism $\coprod_{P(\F)}X_{P(\F)}\to \coprod_Q X_Q$ actually defines a stratification of  $\mathrm{sing}(\F)$.
\end{rmk}

\section{Singularities}\label{singi}
Theorem \ref{teo} gives a condition for a flat family of integrable Pfaff systems to give rise to a flat family of involutive distributions in terms of the flatness of the singular locus.
 We have then to be able to decide when can we apply the theorem. 
More precisely, say 
\[  0\to I(\F)\to \Omega^1_{X|S}\to  \Omega^1_{\F}\to 0\]
is a flat family of codimension $1$ integrable Pfaff systems, and let $s\in S$. 
How do we know when $\mathrm{sing}(\F)$ is flat around $s$? 
In this section we address this question and give a sufficient condition for $\mathrm{sing}(\F)$ to be flat at $s$ in terms of the classification of singular points of the Pfaff system $ 0\to I(\F)_s\to \Omega^1_{X_s}\to  \Omega^1_{\F_s}\to 0$.
\newline
\par From now on, we will only consider Pfaff systems such that $\Omega^1_{\F}$ is torsion-free.  

Remember that, if we have a Pfaff system of codimension $1$, \break $ 0\to I(\F)_s\to \Omega^1_{X_s}\to  \Omega^1_{\F_s}\to 0$, such that  $  \Omega^1_{\F_s}$ is torsion-free, we can consider, locally on $X$, that is given by a single $1$-form $\omega$ and that is integrable iff $\omega\wedge d\omega=0$.

\begin{defs}
Let $ 0\to I(\F)_s\to \Omega^1_{X_s}\to  \Omega^1_{\F_s}\to 0$ be a codimension $1$ integrable Pfaff system. 
And let $x\in X$. 
We say that $x$ is a \emph{Kupka singularity} of the Pfaff system if for some (eq. any) $1$-form $\omega$ locally defining $I(\F)$ around $x$ we have. 
We also say that $x$ is a Kupka singularity of the foliation defined by $I(\F)$.
\begin{align*}
\omega_x=0\\
d\omega_x\neq 0.
\end{align*}
\end{defs}
\begin{prop}
The set of Kupka singularities of a codimension $1$ foliation, if non-void, have a natural structure of codimension $2$ sub-scheme of $X$.
\end{prop}
\begin{proof}[Proof.]
See \cite{LN}
\end{proof}

\begin{defs}
Let $ 0\to I(\F)_s\to \Omega^1_{X_s}\to  \Omega^1_{\F_s}\to 0$  be as before. 
Say $X$ have dimension $n$ over $\C$. 
Then we call $x\in X$ a \emph{Reeb singularity} if there exist an analytical neighborhood $U$ of $x$ such that $I(\F)$ may be generated by a form $\omega$ with the property that $\omega$ can be written, locally in $U$, in the form $\omega=\sum_{i=1}^n f_idz_i$ with
$f_i(x)=0$ for all $i$, and $(df_1)_x,\dots, (df_n)_x$ are linearly independent in $T^*_xX$.
\end{defs}

\begin{rmk}
Note that Kupka singularities and Reeb singularities are singularities in the sense of \ref{singf} i.e.: they are points in $\mathrm{sing}(\F)$.
\end{rmk}

We now give a version for families of the fundamental result of Kupka.

\begin{prop}\label{kupkar}
Let 
\[  0\to I(\F)\to \Omega^1_{X|S}\to  \Omega^1_{\F}\to 0\]
be a flat family of integrable Pfaff systems of codimension $1$, and let $\Sigma=\mathrm{sing}(\F)\subset X$. 
Let $s\in S$, and $x\in \Sigma_s$ be such that $x$ is a Kupka singularity of  $ 0\to I(\F)_s\to \Omega^1_{X_s}\to  \Omega^1_{\F_s}\to 0$.
 Then, locally around $x$ $I(\F)$ can be given by a \emph{relative} $1$-form $\omega(z,s)\in  \Omega^1_{X|S}$ such that
 \[\omega= f_1(z,s)dz_1+f_2(z,s)dz_2, \] 
i.e.: $\omega$ is locally the pull-back of a relative form $\eta\in \Omega^1_{Y|S}$ where $Y\to S$ is of relative dimension $2$.
\end{prop}
The proof is essentially the same as the proof of the classical Kupka theorem, as in \cite{Med}. 
One only needs to note that every ingredient there can be generalized to a relative setup.
\par For this we note that, as $p:X\to S$ is a smooth morphism, the relative tangent sheaf $T_SX$ is locally free and is the dual sheaf of the locally free sheaf $\Omega^1_{X|S}$.
We note also that, if $v\in T_SX(U)$, and $\omega\in \Omega^1_{X|S}(U)$, the relative Lie derivative $L_v(\omega)$ is well defined by Cartan's formula
\[L^S_v= d_S\iota_v(\omega)+\iota_v(d_S \omega),\]
where $\iota_v(\omega)=<v,\omega>$ is the pairing of dual spaces (and by extension also the map $\Omega^q_{X|S}\to \Omega^{q-1}_{X|S}$ determined by $v$), and $d_S$ is the \emph{relative} de Rham differential.
Also $\Omega^q_{X|S}=\wedge^q\Omega^1_{X|S}$. \newline
\par Finally we observe that, if $p:X\to S$ is of relative dimension $d$ and $X$ is a regular variety over $\C$  of total dimension $n$, a family of integrable Pfaff systems gives rise to a foliation on $X$ whose leaves are tangent to the fibers of $p$. 
Indeed, we can pull-back the subsheaf $I(\F)\subset \Omega^1{X|S}$ by the natural epimorphism
\[ f^*\Omega^1_S\to \Omega^1_X\to \Omega^1_{X|S}\to 0,\]
and get $J=I(\F)+f^*\Omega^1_S\subset \Omega^1_X$, which is an integrable Pfaff system in $X$, determining a foliation $\hat{\F}$. 
As  $f^*\Omega^1_S\subset J$, the leaves of the foliation $\hat{\F}$ are contained in the fibers $X_s$ of $p$. 
\par In the general case, where $p$ is smooth but $S$ and $X$ need not to be regular over $\C$, Frobenius theorem still gives foliations $\F_s$ in each fiber $X_s$. 
Indeed, as $p:X\to S$ is smooth, each fiber $X_s$ is regular over $\C$  and, $\Omega^1_{\F}$ being flat, $I(\F)_s\subset \Omega^1_{X_s}$ is an integrable Pfaff system on $X_s$.

\begin{prop}\label{pkr1}
Let $p:X\to S$ a smooth morphism over $\C$ and 
\[ 0 \to I(\F)\to \Omega^1{X|S}\to \Omega^1_{\F}\to 0 \]
a codimension $1$ flat family of Pfaff systems.
Let $\omega \in \Omega^1_{X|S}(U)$ be an integrable $1$-form  such that $I(\F)(U)= (\omega)$ in a neighborhood $U$  of a point $x\in X$. 
Then $d\omega$ is locally decomposable.
\end{prop}
\begin{proof}[Proof.] As $T_SX= (\Omega^1{X|S})^\vee$, and $\Omega^q_{X|S}=\wedge^q \Omega^1_{X|S}$ we can apply Plücker relations to determine if $d\omega$ is locally decomposable and proceed as in \cite{Med}*{Lemma 2.5}. 
\end{proof}

\begin{lema}\label{lkr1}
Suppose that $d\omega_x\neq 0$.
Consider $\mathcal{G}_s$ the codimension $2$ foliations defined by $d\omega$ in $X_s$.
 In the neighborhood $V$ of $x\in X$ where $\mathcal{G}_s$ is non-singular for every $s$ we have the following.
 The leaves of $\mathcal{G}_s$ are integral manifolds of $\omega|_{X_s}$.
\end{lema}
\begin{proof}[Proof.] We only have to prove that, for every $v\in T_SX$ such that  $\iota_v(d\omega)=0$, then $\iota_v(\omega)=0$.
We can do this exactly as in  \cite{Med}*{Lemma 2.6}.
\end{proof}

\begin{lema}\label{lkr2}
With the same hypothesis as \cref{lkr1}. 
Let $v$ be a vector field tangent to $\mathcal{G}$.
Then the \emph{relative}  Lie derivative of $\omega$ with respect to $v$ is zero.
\end{lema}
\begin{proof}[Proof.] Like the proof of  \cite{Med}*{Lemma 2.7}.
\end{proof}

\begin{lema}\label{lkr3} Same hypothesis as \cref{lkr1} and \ref{lkr2}, then
 $\mathrm{sing}(\omega)$ is saturated by leaves of $(\mathcal{G}_s)_{s\in S}$ (i.e.: take $y\in V$ a zero of $\omega$ such that $p(y)=s$, and $L$ the leaf of $\mathcal{G}_s$ going through $y$. Then the inclusion $L\to V$ factorizes through $\mathrm{sing}(\omega)$).
\end{lema}
\begin{proof}[Proof.] 
We can do this entirely on $X_s$. Then this reduces to  \cite{Med}*{Lemma 2.7}.
\end{proof}

\begin{proof}[Proof of \cref{kupkar}.] 
We can take an analytical neighborhood $V$ of $x\in X$ such that $V\cong U\times D^d$ with $U\subseteq S$ an open set, $D^d$ a complex polydisk, and 
\begin{eqnarray*}
p|V:V\cong U\times D^d &\longrightarrow& U.\\
(s, z_1,\dots,z_d)&\mapsto& s
\end{eqnarray*}•
Also, by Frobenius theorem, we can choose the coordinates $z_i$ in such a way that 
\[ v_i=\frac{\partial}{\partial z_i}\in T_SX(V),\qquad 3\leq i\leq d, \]
are tangent to $d\omega$.
Then, as $L^S_{v_i}\omega=0$ and $\iota_{v_i}\omega=0$, we can write $\omega$ as
\[ \omega=f_1(z,s)dz_1+f_2(z,s)dz_2 .\]
\end{proof}

\begin{prop}\label{partek}
Let 
\[  0\to I(\F)\to \Omega^1_{X|S}\to  \Omega^1_{\F}\to 0\]
be a flat family of integrable Pfaff systems of codimension $1$, and let $\Sigma=\mathrm{sing}(\F)\subset X$. 
Let $s\in S$, and $x\in \Sigma_s$ be such that $x$ is a Kupka singularity of  $ 0\to I(\F)_s\to \Omega^1_{X_s}\to  \Omega^1_{\F_s}\to 0$.
 Then $\Sigma\to S$ is smooth around $x$. 
\end{prop}
\begin{proof}[Proof.] By \ref{kupkar} above, we can determine $\Sigma$ around $x$ as the common zeroes of $f_1(z,s)$ and $f_2(z,s)$. The condition $d\omega\neq0$ implies $\Sigma$ is \emph{smooth} over $S$ (remember that we are using the \emph{relative} de Rham differential and that means the variable $s$ counts as a constant).
\end{proof}

\begin{prop}\label{parter}
Let 
\[  0\to I(\F)\to \Omega^1_{X|S}\to  \Omega^1_{\F}\to 0\]
be a flat family of integrable Pfaff systems of codimension $1$, $\Sigma=\mathrm{sing}(\F)\subset X$, and 
$s\in S$. Suppose $x\in \Sigma_s$ is such that $x$ is a \emph{Reeb} singularity of $ 0\to I(\F)_s\to \Omega^1_{X_s}\to  \Omega^1_{\F_s}\to 0$. Then $\Sigma\to S$ is \'etal\'e around $x$. 
\end{prop}
\begin{proof}[Proof.]
The condition on $x$ means we can actually give $I(\F)$ locally by a relative $1$-form $\omega \in \Omega^1_{X|S}$, $\omega=\sum_{i=1}^n f_i(z,s)dz_i$, with $n$ the relative dimension of $X$ over $S$ and the $df_i$'s linearly independent on $x$. Then $\Sigma$ is given by the equations $f_1=\dots=f_n=0$ and is therefore \'etal\'e over $S$.
\end{proof}
 With this two proposition we are almost in condition to state our condition for flatness of the dual family, we just need a general 

\begin{lema}\label{partea}
Let $X\xrightarrow{p} S$ be a morphism between schemes of finite type over an algebraically closed field $k$. 
Let $U\subseteq X$ be the maximal open sub-scheme such that $U\xrightarrow{p} S$ is flat, and $s\in S$ a point such that $X_s$ is without embedded components. . 
 If $U_s\subseteq X_s$ is dense, then $U_s=X_s$.

\end{lema}
\begin{proof}[Proof.] By \cref{rebusc} we must check that, for $A$ either a discrete valuation domain or an Artin ring of the form $k[T]/(T^{n+1})$, and every arrow $\mathrm{Spec}(A)\to S$, the pull-back scheme $X_\mathrm{Spec}(A)$ is flat over $ \mathrm{Spec}(A)$. 
In this way the problem reduces to the case where $S=\mathrm{Spec}(A)$.
\par {\bf (i) Case $A$ DVD.} In this case, $A$ being a principal domain, flatness of $X$ over $\mathrm{Spec}(A)$ is equivalent to the local rings $\O_{X,x}$ being torsion-free $A_{p(x)}$-modules for every point $x\in X$ (\cite[IV.1.3]{SGA1} , so it suffices to consider the case $A_{p(x)}=A$.
\par Now, let $f\in \O_{X,x}$ and $J=\mathrm{Ann}_{ A}(f)\subseteq A$. 
Suppose $J\neq (0)$ and consider $V(J)\subseteq  \mathrm{Spec}(A)$, clearly $\mathrm{supp}(f)\subseteq p^{-1}(V(J))\subseteq X$.
 So $U\cap \mathrm{supp}(f)=\varnothing$. 
But then the restriction $f|s$ of $f$ to $X_s$ have support disjoint with $U_s$.
On the other hand 
\[ \mathrm{supp}(f|_s)=\overline{\{\mathfrak{P}_1\}}\cup\dots\cup\overline{\{\mathfrak{P}_m\}}\subseteq X_s,\]
where $\{\mathfrak{P}_1,\dots, \mathfrak{P}_m\}=\mathrm{Ass}(\O_{X_s,x}/(f|_s))\subseteq \mathrm{Ass}(\O_{X_s,x})$.\newline
As $X_s$ is without immersed components, the $\mathfrak{P}_i$'s are all minimal, so $X_k\cap \overline{\mathfrak{P}_i}$ is an irreducible component of $X_k$, but 
\[U_s\cap \overline{\mathfrak{P}_i}=\varnothing.\]
Contradicting the hypothesis that $U_s$ is dense in $X_s$.
\par {\bf (ii) Case $A=k[T]/(T^{n+1})$.} Using \cref{artcrit} works just as the first case taking $f\in \O_{X,x}$ as a section such that $T^nf=0$ but $f\notin T\O_{X,x}$.

\end{proof}

We have already said that, in a Pfaff system, Kupka singularities, if exists, form a codimension $2$ sub-scheme of $X$. We will call $\mathcal{K}(\F)$ this sub-scheme, and $\overline{\mathcal{K}}(\F)$ its closure. 

\begin{teo}\label{papa}
Let 
\[  0\to I(\F)\to \Omega^1_{X|S}\to  \Omega^1_{\F}\to 0\]
be a flat family of integrable Pfaff systems, for $s\in S$ consider the Pfaff system $ 0\to I(\F)_s\to \Omega^1_{X_s}\to  \Omega^1_{\F_s}\to 0$.
If $\mathrm{sing}(\F_s)$ is without embedded components and $\mathrm{sing}(\F_s)=\overline{\mathcal{K}}(\F_s)\cup\{p_1,\dots,p_m\}$   where the $p_i$'s are Reeb-type singularities, then $\mathrm{sing}(\F)\to S$ is flat in a neighborhood of $s\in S$.
 
\end{teo}
\begin{proof}[Proof.]
Indeed, by \cref{partek},  $\mathrm{sing}(\F)$ is flat in a neighborhood of $ \mathcal{K}(\F_s)$, and, as $\mathrm{sing}(\F_s)$ is without embedded components, we can apply \cref{partea} to conclude that  $\mathrm{sing}(\F)$ is flat in a neighborhood of $\overline{ \mathcal{K}(\F_s)}$.
\par Lastly, from \cref{parter}, follows that $\mathrm{sing}(\F)$ is flat in a neighborhood of $\{p_1,\dots,p_m\}$.
\end{proof}
\section{Applications}\label{apps}

Let $X=\PP^n(\C)$. It's well known that the class of sheaves $\mathscr{F}$ that splits as a direct sum of line bundles $\mathscr{F}\cong\bigoplus_i \O(k_i)$ have no non-trivial deformations. Indeed, as deformation theory teach us, first order deformations of   $\mathscr{F}$ are parametrized by $\mathrm{Ext}^1(\mathscr{F},\mathscr{F})$, in this case we have 
\begin{align*}
\mathrm{Ext}^1(\mathscr{F},\mathscr{F})&\cong \bigoplus _{i,j} \mathrm{Ext}^1(\O(k_i),\O(k_j))\cong\\
\cong  \bigoplus _{i,j} \mathrm{Ext}^1(\O,\O(k_j-k_i))&\cong  \bigoplus _{i,j}  H^1(\PP^n, \O(k_j-k_i))=0
\end{align*}
In particular, given a flat family of distributions 
\[
 0\to T\F\to T_S (\PP^n\times S)\to N_\F\to 0,
\]
such that, for some $s\in S$, $T\F_s \cong \bigoplus_i \O(k_i)$, then the same decomposition holds true for the rest of the members of the family.
\par When we deal with codimension $1$ foliations it's more common, however, to work with Pfaff systems or, more concretely, with integrable twisted $1$-forms $\omega \in \Omega^1_{\PP^n}(d)$, $\omega\wedge d \omega=0$ (see \cite{LN}). It's then than the following question emerged: Given a form  $\omega \in \Omega^1_{\PP^n}(d)$ such that the vector fields that annihilate $\omega$ generate a split sheaf (i.e.: a sheaf that decomposes as direct sum of line bundles), will the same feature hold for every deformation of $\omega$? Such question was addressed by Cukierman and Pereira in \cite{split}. Here we use our results to recover the theorem of Cukierman-Pereira as a special case.

\par
As was observed before, every time we have a codimension $1$ Pfaff system
\[  0\to I(\F)\to \Omega^1_{X}\to  \Omega^1_{\F}\to 0\]
such that $ \Omega^1_{\F}$ is torsion-free, then $ I(\F)$ must be a line bundle.
 In the case $X=\PP^n(\C)$, then $I(\F)\cong \O_{\PP^n}(-d)$ for some $d\in \Z$. 
Is then equivalent to give a Pfaff system and to give a morphism $0\to \O_{\PP^n}(-d)\to  \Omega^1_{\PP^n}$ which is in turn equivalent to $0\to \O_{\PP^n}\to  \Omega^1_{\PP^n}(d)$ that is, to give a global section $\omega$ of the sheaf $\Omega^1_{\PP^n}(d)$.
\par We can explicitly write such a global section as 
\[ \omega=\sum_{i=0}^n f_i(x_0,\dots,x_n)dx_i\]
with $f_i$ a homogeneous polynomial of degree $d-1$ and such that $\sum_i x_i f_i=0$. 
\par Such an expression gives rise to a foliation with split tangent sheaf if and only if there are $n-1$ polynomial vector fields
\begin{align*}
X_1&=g_1^0\frac{\partial}{\partial x_0}+\dots+g_1^n\frac{\partial}{\partial x_n},\\
&\vdots \\
X_{n-1}&=g_{n-1}^0\frac{\partial}{\partial x_0}+\dots+g_{n-1}^n\frac{\partial}{\partial x_n};
\end{align*}
such that $\omega(X_i)=0$, for all $1\leq i\leq n-1$, moreover on a generic point the vector fields must be linearly independents.
\par The singular set of this foliation is given by the ideal $I=(f_0,\dots,f_n)$. 
The condition $\omega(X_i)=0$ means that the ring $ \C[x_0,\dots,x_n]/I$ admits a syzygy of the form

\begin{align*}
& 0\to\C[x_0,...,x_n]^{n}\xrightarrow{\left(\begin{array}{ccc} x_0 &\cdots& x_n\\ g_1^0 &\cdots &g_1^n\\ \vdots&\ddots&\vdots\\  g_{n-1}^0&\cdots&g_{n-1}^n\end{array}\right)} \C[x_0,\dots,x_n]^{n+1}\xrightarrow{(f_0,...,f_n)} \\
&\xrightarrow{(f_0,...,f_n)}  \C[x_0,\dots,x_n]\to \C[x_0,\dots,x_n]/I\to 0 . 
\end{align*}
For such rings a theorem of Hilbert and Schaps tells us the following.

\begin{teo}[Hilbert, Schaps]
Let $A=k[x_0,...x_n]/I$ be such that there is a 3-step resolution of $A$ as above by free modules. Then the ring $A$ is Cohen-Macaulay, in particular, is equidimensional.
\end{teo}
\begin{proof}[Proof.]
This is theorem 5.1 in \cite{artin}
\end{proof}

We thus recover the theorem of Cukierman and Pereira (\cite[Theorem 1]{split}).

\begin{teo}[\cite{split}]
Let 
\[  0\to I(\F)\to \Omega^1_{\PP^n\times S|S}\to  \Omega^1_{\F}\to 0\]
be a flat family of codimension $1$ integrable Pfaff systems. 
And suppose $ 0\to I(\F)_s\to \Omega^1_{\PP^n}\to  \Omega^1_{\F}\to 0$ define a foliation with split tangent sheaf. 
If $\mathrm{sing}(\F_s)\setminus \overline{\mathcal{K}(\F_s)}$ have codimension greater than $2$, then every member of the family defines a split tangent sheaf foliation.
\end{teo}
\begin{proof}[Proof.]
By the above theorem $\mathrm{sing}(\F_s)$ is equidimensional. 
The singular locus of a foliation on $\PP^n$ always have an irreducible component of codimension $2$ (see \cite[Teorema 1.13]{LN}), if  $\mathrm{sing}(\F_s)\setminus \overline{\mathcal{K}(\F_s)}$ have codimension greater than $2$, then it must be empty. 
So  $\mathrm{sing}(\F_s)= \overline{\mathcal{K}(\F_s)}$ and we can then apply \cref{papa}. 
So the flat family 
\[  0\to I(\F)\to \Omega^1_{\PP^n\times S|S}\to  \Omega^1_{\F}\to 0\]
gives rise to a flat family
\[0\to T\F\to T_S X\to N_\F\to 0,\]
so $T\F$ must be flat over $S$, and then $T\F_s$ splits, for every $s\in S$.
\end{proof}

\begin{rmk}
In \cite{split}{Theorem 2} a similar statement is proved for arbitrary codimensional distribution (not necessarily involutive). 
More concretely they prove: 
\par
\emph{Let $D$ be a singular holomorphic distribution on $\PP^n$. If $\mathrm{codim}\ \sing(D)\geq 3$ and} 
\[ TD\cong \bigoplus_{i=1}^d \O(e_i),\qquad e_i\in \Z,\]
\emph{then there exist a Zariski-open neighborhood $\mathcal {U}$ of the space of distribution such that $TD'\cong \bigoplus_{i=1}^d \O(e_i)$ for every $D' \in \mathcal{U}$.}
\par 
When we try to arrive at the above statement as a particular case of the the theory hereby exposed we run into some difficulties, which we will explain now.

Note that generic (non-involutive) singular distributions have isolated singularities.
So, in order to apply an analogue of \cref{papa} to non-involutive distributions of arbitrary codimension, we should be able to conclude, from the hypothesis $\mathrm{codim}\ \sing(D)\geq 3$, that $D$ have isolated singularities.
The author was unable to do so.
It seems that a keener knowledge of the singular set of arbitrary codimensional distributions is necessary to correctly understand and generalize Theorem 2 in \cite{split}. 
\end{rmk}

\end{document}